\documentclass[10pt]{amsart}
\usepackage{amsfonts,amsmath,amssymb,amsthm, mathrsfs}
\usepackage{graphicx,epsfig}
\usepackage{bbm}
\usepackage{t1enc}
\usepackage{subfig}
\usepackage{hyperref}
\usepackage{algorithm}
\usepackage{algorithmic,hyperref}
\usepackage{a4wide}

\usepackage{pstricks-add}
\theoremstyle{plain}

\numberwithin{equation}{section}
\newtheorem{theorem}{Theorem}[section]
\newtheorem{proposition}[theorem]{Proposition}
\newtheorem{corollary}[theorem]{Corollary}
\newtheorem{lemma}[theorem]{Lemma}
\newtheorem{definition}[theorem]{Definition}
\newtheorem{remark}[theorem]{Remark}

\usepackage{color}
\definecolor{darkred}{rgb}{0.8,0,0}
\definecolor{darkblue}{rgb}{0,0,0.7}
\definecolor{darkgreen}{rgb}{0,0.4,0}

\newcommand{\EEE}{\color{black}} 

\newcommand{\eps}{\varepsilon}

\renewcommand{\ln}{\log}

\newcommand{\un}{{\rm 1\kern -2.5pt l}}

\def\vv{\mathbf{v}}

\newcommand{\BBB}{\color{black}}

\def\widthh{6}

\begin{document}
\title{The minimal resistance problem in a class of  non convex bodies}
\author[E. Mainini, M. Monteverde, E. Oudet, D. Percivale]{Edoardo Mainini, Manuel Monteverde, Edouard Oudet and Danilo Percivale}

%
\address{Edoardo Mainini -- Universit\`a degli Studi di Genova,
Dipartimento di
   Ingegneria meccanica, energe\-tica, gestionale
e dei trasporti (DIME), Piazzale Kennedy 1, I-16129 Genova, Italy
}
\email{mainini@dime.unige.it}
\address{Manuel Monteverde -- Universit\`a degli Studi di Genova,
Dipartimento di
   Ingegneria meccanica, energe\-tica, gestionale
e dei trasporti (DIME), Piazzale Kennedy 1, I-16129 Genova, Italy}
\email{monteverde@diptem.unige.it}
\address{\'Edouard Oudet -- Laboratoire Jean Kuntzmann, Universit\'e Grenoble-Alpes,
B\^atiment IMAG, BP 53 38041 Grenoble Cedex 9}
\email{edouard.oudet@imag.fr}
\address{Danilo Percivale -- Universit\`a degli Studi di Genova,
Dipartimento di
   Ingegneria meccanica, energe\-tica, gestionale
e dei trasporti (DIME), Piazzale Kennedy 1, I-16129 Genova, Italy}
\email{percivale@diptem.unige.it}

%
%

\thanks{}

\maketitle

\begin{abstract}
We characterize  the solution to the Newton minimal resistance problem in a class of radial  $q$-concave profiles. We also give the corresponding result for one-dimensional profiles. Moreover, we provide a numerical optimization algorithm for the general nonradial case.

\end{abstract}

\section{Introduction}


A classical problem in the calculus of variations is the minimization of the Newton  functional
\begin{equation*}
D_\Omega(u)=\int_{\Omega}\frac{dx}{1+|\nabla u(x)|^2}.
\end{equation*}
Here,
  $\Omega\subset\mathbb{R}^2$  is a convex set representing the prescribed cross section at the rear end of a body, which  moves with constant velocity through a rarefied fluid in the  orthogonal direction to $\Omega$. The graph of $u:\Omega\to\mathbb{R}$ represents the shape of the body front.
  According to Newton's law the aerodynamic resistance is expressed (up to a dimensional constant) by $D_\Omega$, owing to  the physical assumption of a fluid constituted by independent small particles, each elastically hitting
against the front of the body at most once 
(the so called {\it single shock} property).
As Newton's resistance law is no longer valid when such property
does not hold,
 a relevant design class of profiles for the problem is
$$\mathcal{S}^M(\Omega) =\{u:\Omega\to[0,M]\colon
\text{almost every fluid particle hits the body at most once}\}.
$$
This condition can be rigorously stated as follows:
 for $\Omega$  an open bounded convex subset of $\mathbb{R}^{2}$, we say that $u\colon\Omega\to\mathbb{R}$ is a \textit{single shock function} on $\Omega$ if $u$ is  a.e. differentiable in $\Omega$ and
\begin{equation*}\label{single}u\left(x-\tau\nabla u(x)\right)\le u(x)+\dfrac{\tau}{2}\left(1-|\nabla u(x)|^{2}\right)
\end{equation*}
holds for a.e. $x\in\Omega$ and for every $\tau>0$ such that  $x-\tau\nabla u(x)\in\Omega$, see \cite{BFK2, clr2, P1}. $ \mathcal{S}^M(\Omega)$ is then defined as the class of single shock functions on $\Omega$ that take values in $[0,M]$. The specified maximal cross section $\Omega$ and the restriction on the body length (not exceeding $M>0$) represent given design constraints.

Actually, 
 $\mathcal{S}^M(\Omega)$ lacks of the necessary compactness properties in order to gain  the existence of a global minimizer.
 It is shown  in  \cite{P2} that a minimizer in the class of functions $\mathcal{S}^M(\Omega)$  does not exist and that  the infimum  in this class is
$$\int_{\Omega}\dfrac{1}{2}\left (1-\dfrac{M}{\sqrt{M^{2}+d^{2}(x)}}\right )\,dx,$$
where $d(x)=dist(x,\partial \Omega)$.
This result seem to show that optimal shapes for Newton's aerodynamics can be approximated only by very jagged profiles, practically not to be configured in an engineering project.

Among the different choices in the literature,
 the most classical  set of competing profiles is
\begin{equation*}\label{cm}\mathcal C^{M}(\Omega):=\left\{u:\Omega \to [0, M]:\ \text{ u is concave}\right\},
\end{equation*}
 which automatically implies   the single shock property,  ensures  existence of global minimizers (see \cite{B,BFK2,BG,M}),
  and is more easily configurable. By further assuming radiality, the solution in $\mathcal{C}^M(\Omega)$ ($\Omega$ being a ball in $\mathbb{R}^2$) was described by Newton and it is classically known, see for instance \cite{B,BK,G}. If we reduce the minimization problem in $\mathcal{C}^M(\Omega)$ to the one-dimensional case (i.e., $\Omega$ is an interval in $\mathbb{R}$) the solution is also explicit and easy to determine, see \cite{BK}.
  On the other hand, one of the most interesting features of the Newton resistance functional is the symmetry breaking property, as detected in \cite{BFK1}: the solution among concave functions on a ball in $\mathbb{R}^2$ is not radially symmetric  (and not explicitly known).

The design class $\mathcal{C}^M(\Omega)$ is still  quite restrictive, and
there is a huge gap with the natural class $\mathcal{S}^M(\Omega)$. Indeed, solutions can also be obtained in intermediate classes.  In \cite{clr3,clr2}, existence of  { global} minimizers is shown  among radial profiles in the $W^{1,\infty}_{loc}(\Omega)\cap C^{0}(\bar\Omega)$-closure
of polyhedral functions $u:\Omega\to[0,M]$  ($\Omega$ being a ball in $\mathbb{R}^2$) satisfying the single shock condition.
 In this paper, we are interested in minimizing  the Newton functional in another class of possibly hollow profiles,
  without giving up 
   a complete characterization of one-dimensional and
and radial two-dimensional minimizers. 
  We choose the class of $q$-concave  functions $u$ on $\Omega$ (i.e., $\Omega\ni x\mapsto u(x)-\tfrac q2 |x|^2$ is concave), with height not exceeding the fixed value $M$. That is, 
   given $M>0$ and $q\ge 0$, we let
 $$\mathcal{C}^M_q(\Omega):=\{\left.u\colon \Omega\to [0,M]\right| \,u \text{ is }q\text{-concave on } \Omega\},$$
and we wish to find the minimal resistance among profiles in $\mathcal{C}^M_q(\Omega)$.
We refer to the Appendix at the end of the paper for a discussion about the relation between the two classes $\mathcal{C}^M_q(\Omega)$ and $\mathcal{S}^M(\Omega)$: among $q$-concave functions, the single shock condition is indeed reduced to $q\,\mathrm{diam}(\Omega)\le 2$.
 Of course, for $q=0$ we are reduced to the classical problem in $\mathcal{C}^M(\Omega)$.
If $q>0$, the existence of minimizers is obtained in the same way. However, the characterization of the solution is more involved, even in one dimension ($\Omega$ being an interval in $\mathbb{R}$), and it represents our focus. As a main result we explicitly determine the unique optimal $q$-concave profile, both in the one-dimensional case  and in the radial two-dimensional case, see Section \ref{mainsec} for the statements, under a further {\it high profile} design constraint that we shall introduce therein.

In the one-dimensional case, the symmetry of the solution is not a priori obvious and it is a consequence of our analysis.
On the other hand,
   if $\Omega$ is a ball in $\mathbb{R}^2$ the symmetry breaking  phenomenon appears of course also in the $q$-concave case. 
  When leaving the radial framework, another relevant class is that of developable profiles as introduced in \cite{LP}, playing a role in the numerical approximations \cite{LO} of the optimal resistance.  In Section \ref{numericalsec}, we will  show how to extend the numerical solution of \cite{LO} to the $q$-concave case.

  As a last remark, 
  we notice that large values of $q$ are of course energetically favorable. However,  Newton's law is based on the 
    single shock property which  requires $q\,\mathrm{diam}(\Omega)\le 2$, as previously mentioned. If this restriction is not satisfied,  multiple shock models should be considered as discussed in  \cite{P2}.

\subsection*{Plan of the paper}
In Section \ref{mainsec} we state our two main results.  The first about the one-dimensional case,  $\Omega$ being a line segment. The second deals with the radial two-dimensional case, $\Omega$ being a ball in $\mathbb{R}^2$.  These results were announced in \cite{MMOP}, and they both provide uniqueness of the solution along with an explicit expression. The proofs are postponed to Section \ref{1dsec} and Section \ref{2dsec}, whereas Section \ref{preliminarysec} contains some preliminary results.  Section \ref{numericalsec} provides numerical results for the general $q$-concave two-dimensional problem, i.e., without radiality assumption.
The Appendix contains a discussion about single shock and $q$-concave classes.

\section{Main results}\label{mainsec}

\subsection*{One-dimensional case}
For a locally absolutely continuous function $u:(a,b)\to\mathbb{R}$,  the one-dimensional resistance functional is given by
\begin{equation*}\label{1dfunctional} D_{(a,b)}(u)=\int\limits_{a}^{b}\dfrac{dx}{1+u'(x)^{2}}\,.
\end{equation*}
Without loss of generality we consider the interval $(-1,1)$.  We introduce the variational problem
\begin{equation}\label{qconcprob}\min\limits_{u\in \mathcal{K}^M_q}\int\limits_{-1}^1\dfrac{dx}{1+u'(x)^2}\end{equation}
 for $M>0$ and $q\in [0,1]$,  where $$\mathcal{K}^M_q:=\{\left.u\colon [-1,1]\to [0,M]\right| \,u \text{ is }q\text{-concave on }[-1,1]\}.$$
 Admissible functions $u$ are here $q$-concave on the closed interval $[-1,1]$, meaning that $[-1,1]\ni x\mapsto u(x)-\tfrac q2 x^2$ is concave, and it is  not restrictive to assume they are continuous up to the boundary. \EEE
We will work under the further \textit{high profile} assumption  $2M\ge q$.
The restriction $q\le 1$ corresponds to the { single shock} condition in this case, see Lemma \ref{qshock} in the Appendix. We also refer to the appendix for the standard compactness arguments yielding existence of solutions.
 Our first main result is the following.
\begin{theorem}\label{main1d}
 Let $M> 0$ and $q\in [0,1]$ be such that $2M\ge q$. Then  problem \eqref{qconcprob} has a unique solution  given by
\[
u_{M;q}(x):=\left\{
\begin{array}{lcl}
\left\{
\begin{array}{lcl}
\dfrac q2 (x^{2}-\gamma_{M;q}^{2})+M &\text{if} &|x|\le \gamma_{M;q}\\
\dfrac M{1-\gamma_{M;q}}(1-|x|)&\text{if} &\gamma_{M;q}\le|x|\le 1\\
\end{array}
\right.&\text{if}&M\in (0,1)\vspace{.3cm}\\
M(1-|x|)&\text{if}&M\in [1,+\infty),
\end{array}
\right.\]
where $\gamma_{M;q}\in (0,1)$ is the unique minimizer of the function $R_{M;q}:[0,1]\to\mathbb{R}$ defined by
\begin{equation}\label{RMQ}
R_{M;q}(\gamma)=:\left\{\begin{array}{ll}\dfrac2q\arctan(q\gamma)+\dfrac{2(1-\gamma)^3}{M^2+(1-\gamma)^2}\quad&\mbox{ if $q>0$}\\
2\gamma+\dfrac{2(1-\gamma)^3}{M^2+(1-\gamma)^2}\quad&\mbox{ if $q=0$}.
\end{array}\right.\end{equation}
\end{theorem}

 \begin{figure}[ht]
 \centering
 \begin{tabular}{r}
 \includegraphics[width=7 cm]{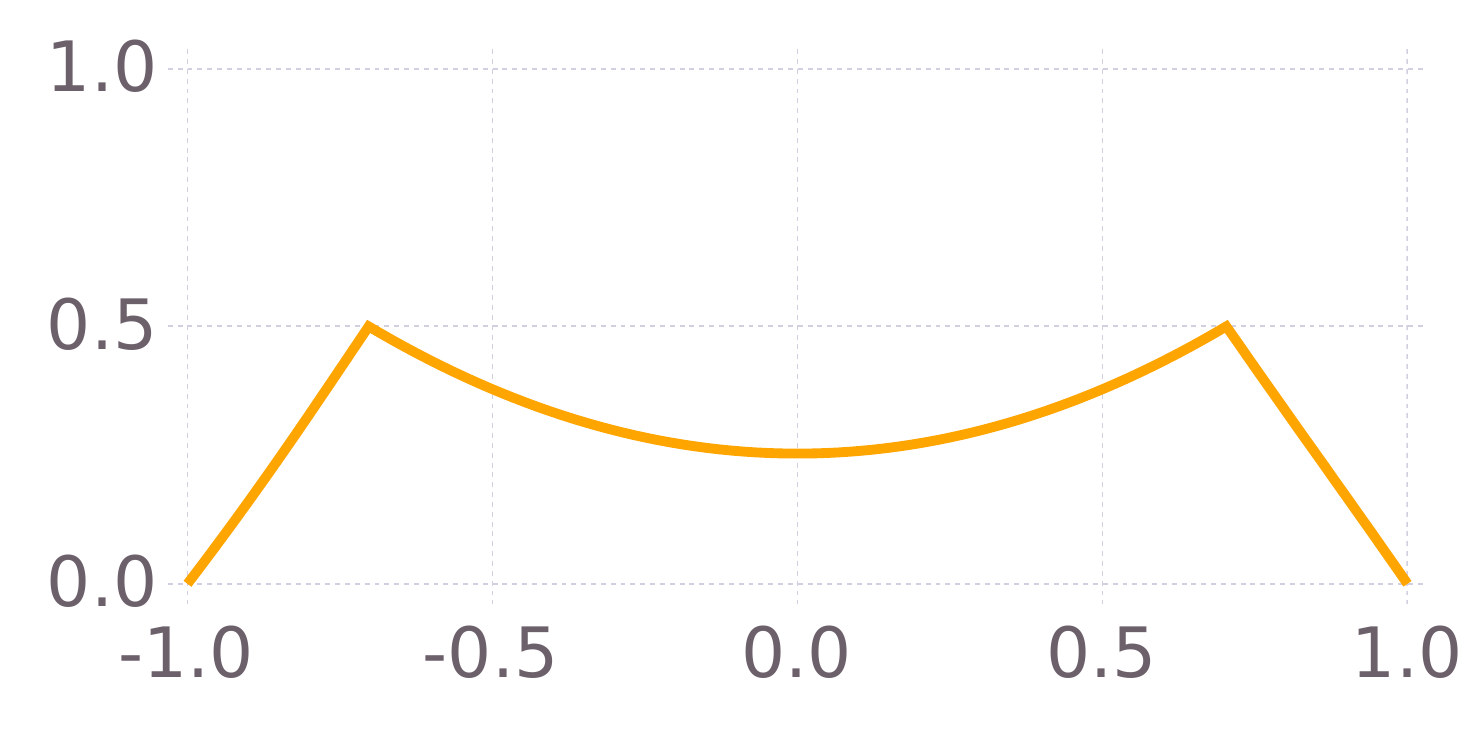}
 \end{tabular}
 \caption{Numerical solution of problem \eqref{qconcprob} for  $M=0.5$ and $q=1$.}
 \label{fig:f0}
 \end{figure}

Theorem \ref{main1d} shows that a solution of problem \eqref{qconcprob}  is given by a piecewise linear and parabolic function (see also the result of a numerical simulation in Figure 1).
Notice that the high profile assumption $2M\ge  q$  ensures that   $u_{M;q}$ fits the interval $[0,M]$ and is therefore admissible for problem \eqref{qconcprob}. The parabolic profile in the center has second derivative equal to $q$. A first understanding of this fact comes from the following straightforward first variation argument.

\medskip
\textbf{Proposition.}
\emph{Let $u$ be a solution to problem \eqref{qconcprob} and suppose that $u\in C^2(I)$ for some open interval $I\subset [-1,1]$. Moreover, suppose that $0<u<M$ in $I$. Then either $u''\equiv0$ or $u''\equiv q$ in $I$}.
\medskip

Indeed,
by $q$-concavity we have $u''\le q$ in $I$.
 Suppose that $u''$ is not identically equal to $q$ in $I$, so that there exists  an open interval $J\subset I$ such that $u''<q$ in $J$. Then, if   $\varphi\in C^\infty_c(J)$ and $|t|$ is small enough, $u+t\varphi$ is still $q$-concave with $0<u+t\varphi<M$ (it is an admissible competitor).
We have by dominated convergence
\[\begin{aligned}
\frac{d}{dt}E(u+t\varphi)&=-2\int_J \frac{\varphi'(x)(u'(x)+t\varphi'(x))}{(1+(u'(x)+t\varphi'(x))^2)^2}\,dx.
\end{aligned}\]
By minimality of $u$ we obtain that for any $\varphi\in C^\infty_c(J)$ there holds
\[
-2\int_J \frac{\varphi'u'}{(1+(u')^2)^2}=2\int_J \frac{u''(1-(u')^2)}{(1+(u')^2)^2}\,\varphi=0,
\]
so that we obtain the standard Euler-Lagrange equation for the Newton functional in one dimension
\[
\frac{u'}{(1+u'^2)^2}=\mathrm{const},
\]
yielding that $u''\equiv 0$ in $J$ and then in $I$.

\subsection*{Radial two-dimensional case}
In this case we let  $\Omega=B_{R}(0)$ be the open ball in $\mathbb{R}^2$, with center $0$ and radius $R>0$, and we consider the class of $q$-concave radial functions.
If we set $M>0$, $q\ge 0$ and
\[\mathcal{R}_{R;M;q}:=\left\{u\colon [0,R]\to [0,M]\left| r\mapsto u(r)-\frac q2r^2 \text{ is nonincreasing  and concave}\right. \right\},\]
then for every $u\in \mathcal{R}_{R;M;q}$ (which  is the radial profile  of a radial function that we still denote by $u$) the resistance functional is 
\begin{equation*}\label{newtonrad}
D_{{B_{R}(0)}}(u)=\mathscr{D}_R(u):=\int\limits_0^R\dfrac{r\,dr}{1+u'(r)^2}.\end{equation*}
Therefore, given $M>0$, $R>0$ and $q\ge 0$,  we have to solve the problem
\begin{equation}\label{radialproblem}
\min\left\{\mathscr{D}_R(u)\colon u\in \mathcal{R}_{R;M;q}\right\},
\end{equation}
still with the high profile assumption $2M\ge qR^2$ and  the single shock assumption $0\le qR\le 1$.
Existence of minimizers  is again standard, see    the Appendix.
Our second  main result is the  characterization of the solution to problem \eqref{radialproblem}.
It is given by a parabolic profile in $[0,a]$, and a strictly decreasing profile satisfying the radial two-dimensional Euler-Lagrange equation
\begin{equation*}\label{eulerad}\dfrac{-ru'(r)}{(1+u'(r)^{2})^{2}}=\mathrm{const}\end{equation*}
in $(a,R]$. The optimal value of $a$ is uniquely determined in $(0,R)$.
In order to write down the solution, which is a little less explicit,  we need to introduce some notation.

We let $(-\infty,-1]\ni t\mapsto h(t):=-t(1+t^{2})^{-2}$.  For $a\in(0,R)$, let $\varphi(a):=-\int _{a}^{R}h^{-1}(\tfrac{a}{4r})\,dr$ and
\[
 \gamma_{q}(a):=\sqrt{\dfrac{1}{2}\big( 3a^{2}q^{2}+1+\sqrt{9a^{4}q^{4}+10a^{2}q^{2}+1}\big)},
\quad \zeta_{q}(a):=-\int_{a}^{R}h^{-1}\left(\frac{a h(-\gamma_{q}(a))} {r}\right)\,dr.
\]


\begin{theorem}\label{radth} Let $R>0$, $M>0$.
Assume that $0\le qR\le1$ and $2M\ge qR^2$.
Then there exists a unique $a_M\in (0,R)$ such that $\varphi(a_M)=M$, and  there exists  a unique $a_{*}\in [a_{M},R)$ such that $\zeta_q(a_*)=M$.
Moreover,  there exists a unique solution to problem \eqref{radialproblem}, given by
  \begin{equation*} u(r):= \left\{
\begin{array}{lcl}
\dfrac {q}{2}(r^{2}-a_{*}^{2})+M&\text{if}&r\in [0, a_{*}]\\
&\\
\displaystyle-\int_{r}^{R}h^{-1}	\left(\frac{a_{*}h(-\gamma_{q}(a_{*}))}{s}\right)\,ds &\text{if}&r\in (a_{*},R].
\end{array}
\right.\end{equation*}
\end{theorem}

It is worth noticing that $\gamma_{0}(a)\equiv 1$, hence when $q=0$ we get $a_{*}=a_{M}$,
 and we recover the classical concave radial minimizer.
 
 Numerical solutions to problem \eqref{radialproblem}, in agreement with Theorem \ref{radth}, are shown in Figure 2. We refer to Section \ref{numericalsec} for numerical solutions obtained without radiality assumption.

   \begin{figure}[h]
 \centering
 \begin{tabular}{r r}
 \includegraphics[width=5 cm]{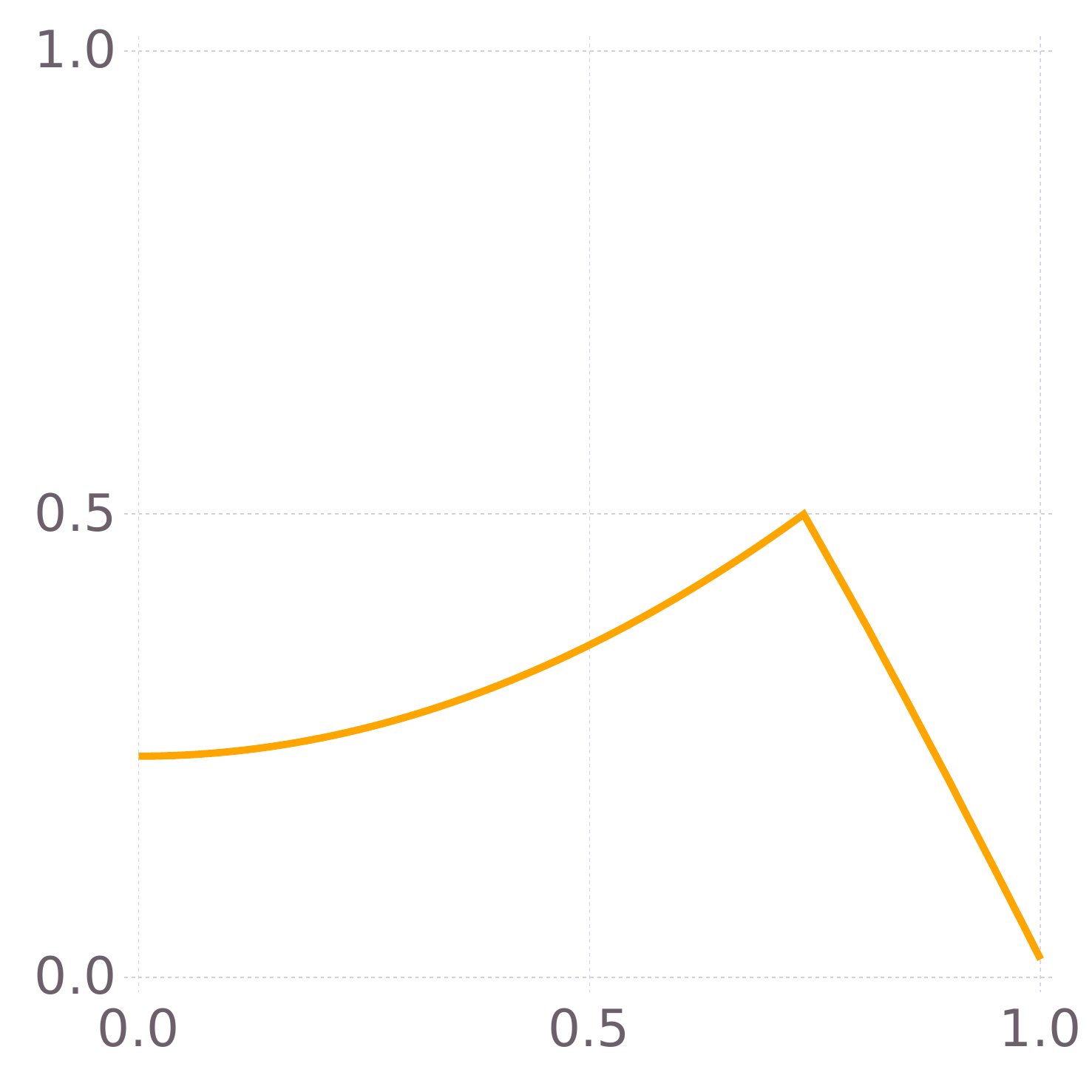}
 &\includegraphics[width=5 cm]{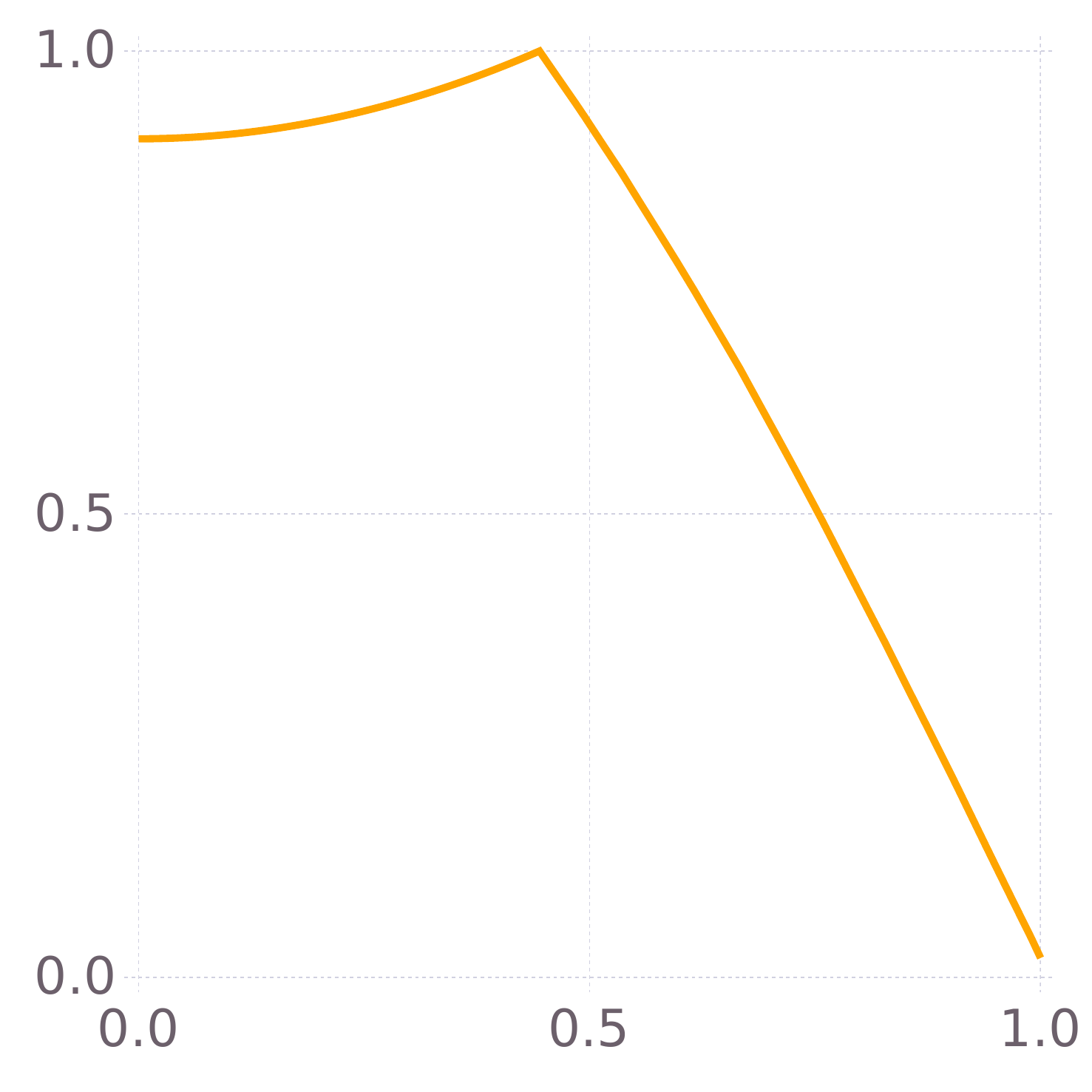}\\
 \includegraphics[width=5 cm]{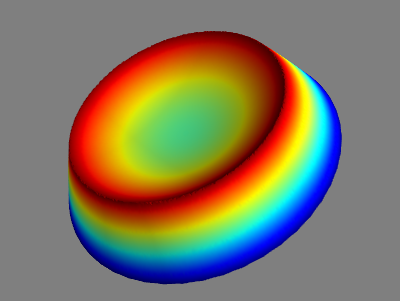}
 &\includegraphics[width=5 cm]{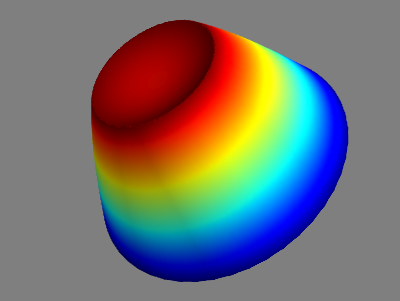}
 \end{tabular}
 \caption{Numerical solutions of problem \eqref{radialproblem} for $M=0.5$ and $M=1$, both  for $R=q=1$.}
 \label{fig:f1}
 \end{figure}



%

\section{Some preliminary results}\label{preliminarysec}

This section gathers some elementary results that will be useful in the sequel. We recall that, for $a<b$ and $q\ge 0$, $u:[a,b]\to\mathbb{R}$ is $q$-concave  if the map $[a,b]\ni x\mapsto u(x)-\tfrac q2 x^2$ is concave.

\EEE

\begin{definition}[\bf Piecewise parabolic approximation]\label{approx}
Let
$a<b$ and $q\ge0$. Let $u$ be a $q$-concave continuous function on $[a,b]$. 
Let $w\colon [a,b]\to \mathbb{R}$ be defined by
\[w(y):=u(y)-\dfrac q2(y-a)(y-b).\] 
For every $h\in\mathbb{N}$ and for every $j\in\left\{0,\ldots,h-1\right\}$ we consider  intervals defined by  $I_{j,h}:=\left[\alpha_{j,h},\beta_{j,h}\right)$
and  $\alpha_{j,h}:=a+j\frac{b-a}{h}$, $\beta_{j,h}:=a+(j+1)\frac{b-a}{h}$.
We let $w_h\colon [a,b]\to \mathbb{R}$ be given by
\begin{equation*}\label{wh}w_h(y):=\left\{\begin{array}{lcl}
\sum\limits_{j=0}^{h-1}\left[
w\left(\alpha_{j,h}\right)+\dfrac{h}{b-a}\left(w\left(\beta_{j,h}\right)-w\left(\alpha_{j,h}\right)\right)(y-\alpha_{j,h})
\right]\mathbbm{1}_{I_j}(y)&\text{if}&y\in[a,b)\\
w(b)&\text{if}&y=b.
\end{array}
\right.\end{equation*}
We define now the sequence of piecewise parabolic approximations  $u_h\colon [a,b]\to\mathbb{R}$ as
\[u_h(y):=w_h(y)+\dfrac q2 (y-a)(y-b),\qquad h\in\mathbb{N}.\]
\end{definition}

\begin{proposition}\label{Dconvergence} Let  $a<b$ and $q\ge 0$. Let $u$ be a $q$-concave continuous function on $[a,b]$.
Let $(u_h)_{h\in\mathbb{N}}$ be the sequence of piecewise parabolic approximations of $u$ given by {\rm Definition \ref{approx}}. Then $D_{(a,b)}(u_h)\to D_{(a,b)}(u)$ as $h\to\infty$.
\end{proposition}
\begin{proof} \BBB
We have $u_h\to u$ uniformly on $[a,b]$ as $h\to\infty$. For any differentiability point $x$ of $u$  which for every $h\in\mathbb{N}$ is not a grid node (that is, for a.e. $x\in (a,b)$), there holds   $u'_h(x)\to u'(x)$. \EEE The result follows by dominated convergence.
\end{proof}

\begin{remark}\label{nonuniform} \rm It is clear that the approximation procedure of Definition \ref{approx} can be generalized to non uniform grids, still with $u_h$ equal to $u$ at grid nodes. Then, uniform convergence, a.e. convergence of derivatives   and the result of  Proposition \ref{Dconvergence} still hold as soon as the maximal size of the grid steps vanishes. In such case, it is possible to let an arbitrarily chosen point in $(a,b)$ be a grid node for any $h$. It is also possible to fix the value of the (right or left) derivative of the approximating sequence at some point. For instance, one may require   $(u'_h)_+(x_0)=u'_+(x_0)$ for any $h$ at some $x_0\in (a,b)$. Indeed, by the monotonicity of $w'_+$,
it is possible to find a sequence of intervals $[x_h,x^h)\ni x_0$, $h\in\mathbb{N}$, such that $x_h\uparrow x_0$ and $x^h\downarrow x_0$ monotonically as $h\to\infty$, and such that $(w(x^h)-w(x_h))/(x^h-x_h)=w'_+(x_0)$ for any $h$. Then, by choosing $x_h,x^h$ to be subsequent grid nodes for the piecewise linear approximation $w_h$ of $w$, the requirement is fulfilled.
\end{remark}

\begin{proposition}[\bf Parallelogram rule]\label{changevar}Let  $\gamma\le\delta$ and  $c\ge 0$. Then
\[\int\limits_\gamma^\delta\dfrac{dx}{1+c(x-\gamma)^2}=\int\limits_\gamma^\delta\dfrac{dx}{1+c(x-\delta)^2}.\]
\end{proposition}
\begin{proof} The thesis follows by the change of variable $x\mapsto \gamma+\delta-x $.
\end{proof}

\begin{proposition}\label{paral}
Let  $a<b$ and  $q\ge 0$. Let $u$ be a $q$-concave function on $[a,b]$ such that $u(a)=u(b)\ge u(x)$ for every $x\in[a,b].$
Then \[\dfrac q2 (x-b)\le u'_+(x)\le u'_-(x)\le \dfrac q2(x-a)\] for every $x\in (a, b).$
\end{proposition}

\begin{proof} Let $x\in (a,b)$ be fixed. Then, by $q$-concavity of $u$ on $[a,b]$, we have that both $u'_+(x)$ and $u'_-(x)$ exist and the following hold
\begin{gather}\label{dx}
u(y)\le u(x)+u'_+(x)(y-x)+\dfrac q2(y-x)^2\text{ for every }y\in [x,b],\\\label{sx}u(z)\le u(x)+u'_-(x)(z-x)+\dfrac q2(z-x)^2\text{ for every }z\in [a,x].\end{gather}
Writing \eqref{dx} for $y=b$ and \eqref{sx} for $z=a$, taking into account that $u(a)=u(b)\ge u(x)$, we get
\[u'_+(x)\ge\dfrac q2(x-b)\quad\text{ and }\quad u'_-(x)\le \dfrac q2(x-a).\]
Moreover, since $x\mapsto u(x)-\tfrac q2x^2$ is a concave function on $[a,b]$, then $u'_-(x)\ge u_+'(x)$ for every $x\in(a,b)$ thus concluding the proof.
\end{proof}

We conclude this preliminary section with the following computation.

\begin{proposition}\label{3var}
Let $\lambda\ge 0$, $F_{\lambda}\colon \mathbb{R}^3\to \mathbb{R}$ be the function defined by
\begin{equation}\label{efflambda}F_{\lambda}(x,y,z):=\arctan x +\arctan y + \arctan z -\arctan\lambda +\arctan(\lambda-x)-\arctan(y+z)\end{equation} and let
\begin{equation}\label{deltalambda}\Delta_{\lambda}:=\left\{(x,y,z)\in\mathbb{R}^3\colon -y\le x\le\lambda,\,-\lambda\le 2y\le 0,\,x-\lambda\le z\le 0\right\}\subseteq\mathbb{R}^3.\end{equation} Then $$\min\limits_{\Delta_{\lambda}}F_{\lambda}= 0. $$
The minimal value is attained if and only if one of the following three cases occurs: $$ \text{i) }\: x=\lambda, z=0, y\in \left[-\tfrac\lambda2,0\right],\quad \text{ii) }\:  x=-y, z=-y-\lambda, y\in \left[-\tfrac{\lambda}2,0\right],\quad \text{iii) }\: x=y=0, z\in[-\lambda,0].$$
\end{proposition}
\begin{proof} If $\lambda=0$ the result is trivial. Let us assume that $\lambda>0$.

We first claim that if $(\overline{x},\overline{y},\overline{z})$ minimizes $F_\lambda$ on $\Delta_\lambda$, then  $\overline{x}=-\overline{y}$ or $\overline{x}=\overline{z}+\lambda.$
Indeed, if $(\overline{x},\overline{y},\overline{z})$ is a minimum point for $F_\lambda$ on $\Delta_\lambda$ satisfying
\begin{equation}\label{strict}
-\overline{y}<\overline{x}<\overline{z}+\lambda,
\end{equation}then it is seen from \eqref{deltalambda} that there exists $\delta>0$ such that $ [\overline{x}-\delta,\overline{x}+\delta]\times\{\overline{y}\}\times\{\overline{z}\}\subseteq\Delta_\lambda$
and \[0=\partial_1 F_\lambda(\overline{x},\overline{y},\overline{z})=\dfrac{\lambda(\lambda-2\overline{x})}{\left(1+\overline{x}^2\right)\left(1+(\lambda-\overline{x})^2\right)},\]
that is $\overline{x}=\tfrac\lambda 2$. Then, from \eqref{deltalambda} and\eqref{strict}  we have
\[-\dfrac\lambda 2<\overline{y}\le 0\text{ and }-\frac\lambda2<\overline{z}\le 0.\]
 If $\overline{x}=\frac\lambda2$, $-\tfrac\lambda 2<\overline{y}< 0\text{ and }-\frac\lambda2<\overline{z}< 0$, then we see from \eqref{deltalambda} and \eqref{strict} that the point  $(\overline{x},\overline{y},\overline{z})$ is in the interior of $\Delta_\lambda$ and therefore
\[\partial_2 F_\lambda\left(\overline{x},\overline{y},\overline{z}\right)=\partial_3 F_\lambda\left(\overline{x},\overline{y},\overline{z}\right)=0,\]
but this is an absurd because the latter equalities hold true  only if $\overline{y}=\overline{z}=0$. Then we are left to consider the case $\overline{x}=\frac\lambda2$,
  $\overline{y}=0$, $-\frac\lambda2<\overline{z}\le 0$ and the case   $\overline{x}=\frac\lambda2$, $-\tfrac\lambda 2<\overline{y}\le 0$, $\overline{z}=0$.
However, in both cases we obtain
\[F_\lambda(\overline{x},\overline{y},\overline{z})=2\arctan\dfrac\lambda 2-\arctan\lambda>0=F_\lambda(0,0,0)\]
and this  contradicts the minimality of $(\overline{x},\overline{y},\overline{z})$, since $(0,0,0)\in \Delta_\lambda.$ The proof of the claim is done,
that is,
 there holds
 $\overline{x}=\overline{z}+\lambda$ or $\overline{x}=-\overline{y}$.

  In order to conclude it suffices to minimize the functions $\varphi_\lambda,\psi_\lambda\colon \mathbb{R}^2\to \mathbb{R}$, defined by
\[\varphi_{\lambda}(y,z):=F(z+\lambda,y,z)=\arctan(z+\lambda)+\arctan y-\arctan \lambda-\arctan(y+z),\]
\[\psi_{\lambda}(y,z):=F(-y,y,z)=\arctan z-\arctan\lambda+\arctan(\lambda+y)-\arctan(y+z),\]
 on the set
\[\Sigma_\lambda:=\left\{(y,z)\in\mathbb{R}^2\colon y\in \left[-\tfrac\lambda 2,0\right],z\in [-\lambda-y,0]\right\}.\]
It is easily seen than both $\varphi_\lambda$ and $\psi_\lambda$ have no critical points in the interior of ${\Sigma_\lambda}$.
Let us check their behavior on the boundary of $\Sigma_\lambda$.

There holds
\begin{equation}\label{equa1}\varphi_\lambda(y,0)=0=\varphi_\lambda(y,-y-\lambda)\quad \mbox{for every $y\in\left[-\tfrac{\lambda}2,0\right]$}.\end{equation}
The restrictions of $\varphi_\lambda$ on the other two edges of the boundary of $\Sigma_\lambda$ are
\[\tilde{\varphi}_\lambda(z):=\varphi_\lambda\left(-\tfrac\lambda 2,z\right)=\arctan(z+\lambda)-\arctan\tfrac\lambda2-\arctan\lambda-\arctan\left(z-\tfrac\lambda2\right),\quad z\in \left[-\tfrac{\lambda}2,0\right]
\]
and
\[\overline{\varphi}_\lambda (z):=\varphi_\lambda(0,z)=\arctan(z+\lambda)-\arctan\lambda-\arctan z,\quad z\in[-\lambda,0].
\]
Then we can see that $\tilde{\varphi}_\lambda$ is strictly increasing in $\left[-\tfrac\lambda 2,-\tfrac\lambda4\right]$ and strictly decreasing in $\left[-\tfrac\lambda4,0\right]$ while $\overline{\varphi}_\lambda$ is strictly  increasing in $[-\lambda,-\tfrac\lambda2]$ and strictly decreasing in $\left[-\tfrac\lambda2,0\right]$. This yields  $\tilde\varphi_\lambda\ge 0$ on $[-\frac\lambda2,0]$ with equality only at $-\frac\lambda2$ and $0$, and   $\overline{\varphi}_\lambda>0$ on $[-\lambda, 0]$ with equality only at $-\lambda$ and $0$. Therefore, $\varphi_\lambda\ge 0$ on $\Sigma_\lambda$, the only equality cases being described by \eqref{equa1}.

 Similarly, $\psi_\lambda(0,z)=0$ for every $z\in[-\lambda,0]$ and $\psi_\lambda(y,-y-\lambda)=0$ for every $y\in \left[-\tfrac{\lambda}2,0\right]$, and moreover  $\psi_\lambda>0$ on the rest of the boundary of $\Sigma_\lambda$. Indeed, after setting
\[\tilde{\psi}_\lambda(z):=\psi_\lambda\left(-\tfrac\lambda 2,z\right)=\arctan z-\arctan\lambda+\arctan\tfrac\lambda2-\arctan\left(z-\tfrac\lambda2\right),\quad z\in\left[-\tfrac\lambda2,0\right]\]
and
\[\overline{\psi}_\lambda(y):=\psi_\lambda(y,0)=-\arctan\lambda+\arctan(\lambda+y)-\arctan y,\quad y\in\left[-\tfrac\lambda2,0\right]
\]
 it is easily seen that $\tilde{\psi}_\lambda$ is strictly increasing in $\left[-\tfrac\lambda2,0\right]$ and $\overline{\psi}_\lambda$ is strictly decreasing on the same interval. The proof is concluded.
\end{proof}

\section{The one-dimensional case}\label{1dsec}

In the following we will make use of the notation
\begin{equation*}\label{pnotation}\wp_{a;b}^K(y):=\dfrac q2 (y-a)(y-b)+K,\quad\text{  }y\in[a,b].\end{equation*}
The proof of  Theorem \ref{main1d} is essentially based on the following Lemma \ref{struccenter} and Lemma \ref{strucside}. The first identifies  the parabolic profile as optimal in the center. The latter identifies a linear profile on the  side.

\begin{lemma}[\bf The center]\label{struccenter}
Let  $a<b$, $q\ge 0$, and let $u$ be a $q$-concave function on $[a,b]$ such that
$u(a)=u(b)\ge u(x)$ for every $x\in[a,b].$
Then \begin{equation*}\label{th}
D_{(a,b)}(u)\ge D_{(a,b)}\left(\wp_{a;b}^{u(a)}\right)\end{equation*}
and equality holds if and only if $u\equiv \wp_{a;b}^{u(a)}$.
\end{lemma}

\begin{proof}
If $q=0$ the result is trivial. Assume therefore that $q>0$.
Since  $v_h(x):=u(x-h)$ satisfies $D_{(a,b)}(u)=D_{(a+h,b+h)}(v_h)$ for any $h\in\mathbb{R}$ (translation invariance property),  we may also assume without loss of generality that the reference interval is of the form $[-a,a]$, $a>0$.

Notice that $u$ is absolutely continuous in $[-a,a]$ and that $u(a)=u(-a)$ entails
$\int_{-a}^{a}u'(x)\,dx=0$,
hence  the set $\{x\in(-a,a):u'_+(x)\le qx\}$ is nonempty and we  may define
\begin{equation}\label{zetagreca}\zeta:=\inf\{x\in(-a,a)\colon u'_+(x)\le qx\}.\end{equation}
Then we have $\zeta\in[-a,a)$, and moreover we may assume without loss of generality that $\zeta\le 0$ (indeed, if this is not the case we may consider $v(x):=u(-x)$, which still satisfies the assumptions, since it is clear that the corresponding value of $\zeta$ is in $[-a,0]$, and since $D_{-a}^a(v)=D_{-a}^a(u)$ obviously holds).
Notice also that $\zeta\le 0$ implies $u'_+(0)\le 0$, since $x\mapsto u'_+(x)-qx$ is nonincreasing.

The proof will be achieved in some steps.
We  first prove  that
\begin{equation}\label{tha}
D_{(-a,a)}(u)\ge D_{(-a,a)}\left(\wp_{-a;a}^{u(a)}\right)
\end{equation}
holds true
for $q$-concave functions $u$, satisfying $u(a)=u(-a)\ge u(x)$ for any $x\in [-a,a]$,   such that  $[-a,a]\ni x\mapsto u(x)-\tfrac q2(x^2-a^2)$ is piecewise linear. In such case $x\mapsto u'(x)-qx$ is a nonincreasing piecewise constant function on $(-a,a)$.
We will consider a general  $u$ only in the last step.

\textbf{Step 1}.
As previously observed, it is not restrictive to assume $\zeta\le0$.
 Let $A_1$ the (possibly empty) set defined by $A_1:=\{x\in(-a,\zeta)\colon  u_+'(x)> 0\},$
 and let $A_2=(-a,\zeta)\setminus A_1.$
  Since $u'$ is piecewise linear, $A_1$ is a finite disjoint union of open  intervals $(c_i,d_i)$, $i=1,\ldots,k$, and
\begin{equation}\label{A1}\begin{aligned}\int\limits_{A_1}\dfrac{dx}{1+u'(x)^2}&\ge\sum\limits_{i=1}^{k}
\int\limits_{c_i}^{d_i}\dfrac{dx}{1+q^2(x-c_i)^2}
\\&=\sum\limits_{i=1}^k
\int\limits_{c_i}^{d_i}\dfrac{dx}{1+q^2(x-d_i)^2}
\ge \sum\limits_{i=1}^k
\int\limits_{c_i}^{d_i}\dfrac{dx}{1+q^2x^2}
=\int\limits_{A_1}\dfrac{dx}{1+q^2x^2}.\end{aligned}
\end{equation}
Here, the first inequality holds true since $x\mapsto u_+'(x)-qx$ (equal to $u'(x)-qx$ a.e. on $(-a,a)$) is not increasing and since $u'_+(c_i)=0$, so that on $(c_i,d_i)$ we have $0<u'_+(x)\le q(x-c_i)$. The first equality follows by Proposition \ref{changevar} and the last inequality is satisfied since we  have $d_i\le 0$ and then $0< q(d_i-x)\le -qx$ on $(c_i,d_i)$, for every $i=1,\ldots,k$. On the other hand, it is clear that we have $0\ge u'_+(x)\ge qx$ on $A_2$
and together with \eqref{A1} this gives
\begin{equation}\label{minuszeta}
D_{(-a,\zeta)}(u)=\int_{A_1\cup A_2}\frac{dx}{1+u'(x)^2}\ge \int_{-a}^\zeta\frac{dx}{1+q^2x^2}.
\end{equation}

In a  similar way, since $u$ is $q$-concave on $[-a,a]$ and $\zeta \in[-a,0]$, we have that
$0\ge qx \ge u'_+(x)\ge u'_+(0)+qx$
for every $x\in(\zeta,0)$. As $u'=u'_+$ a.e. on $(-a,a)$, we get
\begin{equation}\label{zetazero}\begin{aligned}
D_{(\zeta,0)}(u)&\ge\int\limits_\zeta^0\dfrac{dx}{1+(u'_+(0)+qx)^2}=\dfrac1q\arctan(u'_+(0))-\dfrac1q\arctan(u'_+(0)+q\zeta)\\
&=\int\limits_\zeta^0\dfrac{dx}{1+q^2x^2}+\dfrac1q\arctan(u'_+(0))-\dfrac1q\arctan(u'_+(0)+q\zeta)
+\dfrac1q\arctan(q\zeta).
\end{aligned}\end{equation}

\textbf{Step 2}.
Let us now define
 \begin{equation}\label{S}\begin{aligned}
 \sigma:&=\min\{x\in [0,a)\colon u'_+(x)=0\},\\
 S:&=\{x\in (0,a)\colon u'_-(x)=u'_+(x)=0\}\cup \{x\in (0,a)\colon u'_{+}(x)<0<u'_{-}(x)\}.
 \end{aligned}
 \end{equation}
 By Proposition \ref{paral}, we have $\tfrac q2(x-a)\le u'(x)\le\tfrac q2 (x+a)$ at each point  where $u'$ exists. Since $u'(x)-qx$ is piecewise constant, it follows that $-qx>u'(x)-qx\ge 0$ on a right neighbor of $-a$  and $-qx<u'(x)-qx\le 0$ on a left neighbor of $a$. Moreover  $u'_+(0)\le 0$ follows from $\zeta\le 0$. Therefore,  $\sigma$ is well defined. We have $0<\sigma<a$  if $u'_+(0)<0$ (then $u'_+< 0$ on $(0,\sigma)$) and  $\sigma=0$ otherwise. In any case $u'_+(\sigma)=0$. On the other hand,
   $u'$ is piecewise linear, therefore $S$ is a (possibly empty) finite set,
 and sign change of $u'_+$ on $(0,a)$  occurs exactly at $\sigma$ if $\sigma>0$, and  on $S\setminus\{\sigma\}$, if nonempty. In case $S\setminus\{\sigma\}$ is nonempty,  we denote its elements by $0<\xi_1<\xi_2<\ldots< \xi_h$, and $h$ is even (this comes  from the fact that $u'>0$ in a left neighborhood of $a$). We also let $\xi_{h+1}=a$.  In each of the intervals $(\xi_{i},\xi_{i+1})$, $i=1,\ldots, h$, there holds either $u'_+\ge 0$ or $u'_+\le 0$.
Moreover we have that
\begin{equation}\label{xyz}D_{(\xi_i,\xi_{i+1})}(u)\ge \int\limits_{\xi_i}^{\xi_{i+1}}\dfrac{dx}{1+q^2(x-\xi_i)^2}
\end{equation}
for every $i=1,\ldots,h$. Indeed,  \eqref{xyz} is obvious if $u'_+\ge 0$ on $(\xi_i,\xi_{i+1})$, i.e., $u'\ge 0$ a.e. on $(\xi_i,\xi_{i+1})$. Else if   $u'\le 0$ a.e. on $(\xi_i,\ \xi_{i+1})$, the $q$-concavity inequality $0\ge u'_+(x)\ge q(x-\xi_{i+1})$ and Proposition \ref{changevar} yield
\[D_{(\xi_i,\xi_{i+1})}(u)\ge
\int\limits_{\xi_i}^{\xi_{i+1}}\dfrac{dx}{1+q^2(x-\xi_{i+1})^2}=\int\limits_{\xi_i}^{\xi_{i+1}}\dfrac{dx}{1+q^2(x-\xi_{i})^2}. \]
If instead $S\setminus \{\sigma\}$ is empty we just have $h=0$ and $\xi_1=a$.
 Similarly,  $q$-concavity implies $0\le u'_+(x)\le q(x-\sigma)$ on $(\sigma,\xi_1)$, and in case  $\sigma>0$ it gives
   $q(x-\sigma)\le u'_+(x)\le0$ on $(0,\sigma)$. Then the usual change of variables of Proposition \ref{changevar} entails
\begin{equation}\label{zerosigma}
D_{(0,\sigma)}(u)\ge \int\limits_0^\sigma \dfrac{dx}{1+q^2x^2}, \qquad D_{(\sigma,\xi_1)}(u)\ge  \int\limits_\sigma^{\xi_1} \dfrac{dx}{1+q^2(x-\sigma)^2}.
\end{equation}
In general,
from \eqref{xyz} and \eqref{zerosigma} we have
\[\begin{aligned}
D_{(0,a)}(u)&=D_{(0,\sigma)}(u)+D_{(\sigma,{\xi_1})}(u)+\sum_{i=1}^hD_{(\xi_i,\xi_{i+1})}(u)\\&\ge
\int\limits_0^\sigma \dfrac{dx}{1+q^2x^2}+\int\limits_\sigma^{\xi_1} \dfrac{dx}{1+q^2(x-\sigma)^2}+
\sum\limits_{i=1}^{h}\int\limits_{\xi_i}^{\xi_{i+1}} \dfrac{dx}{1+q^2(x-\xi_i)^2} \\
&=\int\limits_0^\sigma \dfrac{dx}{1+q^2x^2}+ \dfrac1q\arctan(q\xi_1-q\sigma)+\sum_{i=1}^{h}\dfrac1q\arctan(q\xi_{i+1}-q\xi_i).
\end{aligned}\]
The sub-additivity of $	\arctan$ in $\mathbb{R}_+$ then implies
\begin{equation}\label{sigmaa}\begin{aligned}
D_{(0,a)}(u)&\ge\int\limits_{0}^{\sigma}\dfrac{dx}{1+q^2x^2}+\dfrac1q\arctan(qa-q\sigma)
\\&=\int\limits_{0}^{a}\dfrac{dx}{1+q^2x^2}+\dfrac1q\arctan(qa-q\sigma)-\dfrac1q\arctan(qa)+\dfrac1q\arctan(q\sigma).
\end{aligned}\end{equation}

\textbf{Step 3}.
Adding together \eqref{minuszeta}, \eqref{zetazero} and \eqref{sigmaa}  we get
\begin{equation}\label{Dlast}
D_{(-a,a)}(u)\ge
\int\limits_{-a}^a\dfrac{dx}{1+q^2x^2}
+\dfrac1qF_{qa}(q\sigma, u'_+(0),q\zeta)
\end{equation}
where $F_{qa}$ is the function defined in \eqref{efflambda} with $\lambda=qa>0$,
so that in order to conclude it is enough to show that
\begin{equation}\label{app}(q\sigma, u'_+(0),q\zeta)\in \Delta_{qa},
\end{equation}
being $\Delta_{qa}$  the set  defined in \eqref{deltalambda} with $\lambda=qa$, and then apply Proposition \ref{3var}.

We already observed that $q\sigma\le qa$, $u'_+(0)\le 0$ and $q\zeta\le 0$. Moreover, $q$-concavity and $u'_+(\sigma)=0$ yield  $u'_+(0)\ge u'_+(\sigma)-q\sigma=-q\sigma$. Since $u(-a)=u(a)\ge u(x)$ for every $x\in [-a,a]$, by applying Proposition \ref{paral} we obtain that $2u'_+(0)\ge -qa.$
At last we claim that $q\sigma-qa\le q\zeta.$ Indeed we have $\int_o^\sigma u'(x)dx\le 0$ and $u'(x)\le q(x-\sigma)$ a.e. on $(\sigma,a)$, whereas $u(a)=u(-a)\ge u(\zeta)$ by assumption, thus
\[\begin{aligned}
0&=\int\limits_{-a}^au'(t)\,dt=\int\limits_{-a}^\zeta u'(t)\,dt+\int\limits_{\zeta}^au'(t)\,dt=u(\zeta)-u(a)+\int\limits_{\zeta}^au'(t)\,dt\le\int\limits_{\zeta}^au'(t)\,dt\\&
=\int\limits_{\zeta}^0u'(t)\,dt+\int\limits_{0}^\sigma u'(t)\,dt+\int\limits_{\sigma}^au'(t)\,dt\le \int\limits_{\zeta}^0qt\,dt+\int\limits_{\sigma}^a q(t-\sigma)\,dt
=\dfrac{1}2(a-\sigma-\zeta)(qa-q\sigma+q\zeta),
\end{aligned}\]
but $a-\sigma\ge0$ and $\zeta \le0$ then the claim is proved, and \eqref{app} is shown, so that \eqref{Dlast} and Proposition \ref{3var} allow to conclude that
\[D_{(-a,a)}(u)\ge \int\limits_{-a}^a\dfrac{dx}{1+q^2x^2}=D_{(-a,a)}\left(\wp_{-a;a}^{u(a)}\right),\]
in case $x\mapsto u'(x)-qx$ is piecewise constant.

\textbf{Step 4}.
In order to treat a general $q$-concave function $u$, satisfying $u(-a)=u(a)\ge u(x)$ for any $x\in [-a,a]$, we approximate it by means of the sequence $u_h$ from Definition \ref{approx}.
Then, \eqref{tha} applies to $u_h$ for each $h$, as just shown. Invoking Proposition \ref{Dconvergence}, we find \eqref{tha} for $u$.

We are left to prove that the only equality case in \eqref{tha} is $u=\wp_{-a;a}^{u(a)}$, i.e., $u'(x)=qx$ in $(-a,a)$.
This is done by revisiting the previous steps and by taking some care in the choice of the approximating sequence $u_h$.
Assume that $u$ satisfies \eqref{tha} with equality.
As usual,  we  may assume  that the number $\zeta$ defined by \eqref{zetagreca} is nonpositive, then  $u'_+(0)\le 0$. If $\zeta=-a$, then $\int_{-a}^a u'=0$ readily implies $u'=qx$ on $(-a,a)$. Therefore, we assume that $\zeta>-a$ as well, and we aim at reaching a contradiction.

We first claim that $u'_+\le 0$  in the whole $(0,a)$ yields contradiction: indeed, it would give,  by taking into account that $u(x)\le u(-a)$ in $[-a,a]$ and that $u'\le 0$ a.e. on $(\zeta,0)$,
\begin{equation*}
0=\int\limits_{-a}^au'(x)\,dx\le\int\limits_{\zeta}^au'(x)\,dx
=\int\limits_{\zeta}^0u'(x)\,dx+\int\limits_{0}^a u'(x)\,dx\le \int\limits_{\zeta}^0qx\,dx\le -\dfrac{q}{2}\zeta^{2},
\end{equation*}
that is, $\zeta=0$.
\BBB But $\zeta= 0$ implies $D_{(-a,0)}(u)\ge D_{(-a,0)}\left(\wp_{-a;a}^{u(a)}\right)$: this follows from Step 1, see \eqref{A1} and \eqref{minuszeta},
where in this case the set $A_1$ is a possibly infinite but countable union of disjoint open intervals (because $A_1$ is open, since $u'_+$ is lower semicontinous).
\EEE
 On the other hand, Proposition \ref{paral} implies $u'(x)\ge \tfrac q2 (x-a)$ a.e. on $(0,a)$, then  $u'\le 0$  gives $u'(x)^2\le \tfrac{q^2}4(x-a)^2$ and  Proposition \ref{changevar} yields
\[
D_{(0,a)}(u)=\int_0^a\frac{dx}{1+u'(x)^2}\ge\int_0^a\frac{dx}{1+\tfrac {q^2} 4(x-a)^2}=\int_0^a\frac{dx}{1+\tfrac {q^2} 4 x^2}>
D_{(0,a)}\left(\wp_{-a;a}^{u(a)}\right),
\]
that is, summing up, $D_{(-a,a)}(u)> D_{(-a,a)}(\tilde u)$, a contradiction.
The  claim is proved and thus we assume from now that $u'_+>0$ at some point in $(0,a)$, which implies, by $q$-concavity of $u$ and  right continuity of $u'_+$, that
$ \sigma$ from \eqref{S}
is  well defined for $u$, with $u'_+(\sigma)=0$ and $-a<\zeta\le0\le\sigma<a$.

We approximate $u$ with a sequence of $q$-concave piecewise parabolic functions  $u_{h}$, constructed by means of Remark \ref{nonuniform},
 such that $u_h(\pm a)=u(\pm a)$,  $ (u_{h})'\to  u'$ a.e. on $(-a,a)$ and 
 \begin{equation}\label{points}
  (u_{h})'_{{+}}(\sigma)= u'_{+}(\sigma),\quad
  (u_{h})'_{{+}}(\zeta)= u'_{+}(\zeta),\quad
     (u_{h})'_{{+}}(0)= u'_{+}(0), \quad \forall h\in\mathbb{N}.
     \end{equation}
 \EEE We let  $\zeta_h:=\inf\{x\in(-a,a)\colon (u_h)'_+(x)\le qx\}$.
 By definition of $\zeta_h$ and $\zeta$ and by \eqref{points},  we see that $\zeta_h\le\zeta$ and that $\zeta_h\to\zeta$ as $h\to \infty$. We let $\sigma_h:=\min\{x\in [0,a)\colon (u_h)'_+(x)=0\}$, then \eqref{points} implies $\sigma_h\le\sigma$. Notice that if $u'_+(0)=0$, then $\sigma=0$ so that $\sigma_h=0$ for any $h$. Else if $(u_h)'_+(0)=u'_+(0)<0$ we have by $q$-concavity $(u_h)'_+(x)\le qx+u'_+(0)$ on $[0,a)$, implying  $q\sigma_h\ge-{(u_h)'_+(0)}=-u'_+(0)$. Therefore $\sigma_h\in\left[-{u'_+(0)}/{q},\sigma\right]$, and we may assume, up to passing on a not relabeled subsequence, that $\sigma_h\to\bar\sigma\in\left[-{u'_+(0)}/{q},\sigma\right]$ as $h\to\infty$.

We apply the previous steps  obtaining \ref{Dlast} for $u_h$, and passing to the limit with the a.e. convergence of $u'_h$ to $u'$ and with the continuity of function $F_{qa}$ we get
\begin{equation*}
D_{(-a,a)}(u)\ge
\int\limits_{-a}^a\dfrac{dx}{1+q^2x^2}
+\dfrac1qF_{qa}(q\bar\sigma, u'_+(0),q\zeta).
\end{equation*}
If $F_{qa}(q\bar\sigma, u'_+(0),q\zeta)>0$ we contradict the fact that $u$ satisfies \eqref{tha} with equality.
By taking into account that $\bar\sigma\le \sigma< a$,  Proposition \ref{3var} shows that $F_{qa}(q\bar\sigma, u'_+(0),q\zeta)=0$
if and only if one of the following two cases occurs
\[
{\rm i)}\;  0<u'_{+}(0)=-q\bar\sigma,\ \ \zeta=\bar\sigma-a,\qquad
{\rm ii)}\; \bar\sigma=0=  u'_{+}(0).
\]

If i) were true then $ u'_+(x)\le qx+u'_+(0)=q(x-\zeta-a)$ for every $x\in (0,a)$, hence by taking into account that $u(\zeta)\le u(-a)$ we would get
\[0=\int_{-a}^{a} u'(x)\,dx\le \int_{\zeta}^{0} qx\,dx+\int_{0}^{a}q(x-\zeta-a)\,dx=-\dfrac{q}{2}(\zeta+a)^{2}\]
that is $\zeta=-a$, a contradiction.

Eventually if ii) occurs then we are in the case $\sigma_h=\sigma=0$. \BBB In this case it is clear that $u_+'(x)-qx$, which is monotone, is identically $0$ on $(\zeta,0)$, and moreover  we immediately get
\begin{equation}\label{ecco}D_{(-a,0)}\left(\wp_{-a;a}^{u(a)}\right)\le D_{(-a,0)}(u),
\end{equation}
 since equality holds on $(\zeta,0)$ where $u'(x)\equiv qx$, and since we apply Step 1 on $(-a,\zeta)$, recalling as before that in general the set $A_1$ therein is a countable union of disjoint open intervals.

 \EEE If $0\le u'_+(x)\le qx$ in $(0,a)$, either $u'=qx$ a.e. in $(0,a)$,  thus in $(\zeta,a)$, and then we easily see from the null mean property of $u'$ that $\zeta=-a$ (a contradiction), or $u'=qx$ does not hold a.e. in $(0,a)$ \BBB and  we readily conclude that
$D_{(0,a)}\left(\wp_{-a;a}^{u(a)}\right)< D_{(0,a)}(u)$, which, combined with \eqref{ecco},  yields that \eqref{tha} does not hold with equality, a contradiction. \EEE

 Else if $u'_+<0$ at some point  $c\in(0,a)$,
since we are also excluding $u'_+\le 0$ on the whole $(0,a)$, \BBB we  also fix a point  $d\in (0,a)$ such that $u'_+(d)>0$.  In this case, we assume that the  above approximating sequence $u_h$  satisfies  a further restriction, still by means of Remark \ref{nonuniform}: we let $(u_{h})'_{{+}}(c)= u'_{+}(c)$ and
$  (u_{h})'_{{+}}(d)= u'_{+}(d)
$ for any $h\in\mathbb{N}$. Therefore, after defining
$$S_h:=\{x\in (0,a)\colon (u_h)'_-(x)=(u_h)'_+(x)=0\}\cup \{x\in (0,a)\colon (u_h)'_{+}(x)<0<(u_h)'_{-}(x)\},$$
it is clear that for any $h\in\mathbb{N}$ there is an element $\bar x_h$ in the set $S_h\cap[c\wedge d, c\vee d]$. Indeed, $u_h'$ has to change sign at least once on $[c\wedge d, c\vee d]$.
Now
 we can reason as in Step 2.
 \BBB  Fix  $h\in\mathbb{N}$.
 Let $0=\xi_0<\xi_1<\ldots< \xi_{n}=\bar x_h$ and $\bar x_h=\xi_{n+1}\ldots<\xi_{n+m-1}$ denote the finitely many points of $S_h$, and let $\xi_{n+m}=a$ ($S_h$ contains at least $\bar x_h$). Since \eqref{xyz} holds for $u_h$ in any of the intervals $(\xi_i,\xi_{i+1})$, where $u'_h$ does not change sign, we get
\[
\begin{aligned}
D_{(0,a)}(u_h)&\ge\sum_{i=0}^{n+m-1}\int_{\xi_i}^{\xi_{i+1}}\frac{dx}{1+q^2(x-\xi_i)^2}=
\sum_{i=0}^{n+m-1}\tfrac 1q\arctan{(q(\xi_{i+1}-\xi_i))}
\\
&\ge \frac 1q\arctan(q\bar x_h)+\frac 1q\arctan(q(a-\bar x_h))
\end{aligned}
\]
where we have split the sum and used the sub-additivity of arctan. By passing to the limit  with Proposition \ref{Dconvergence} and Remark \ref{nonuniform} as $h\to \infty$ (possibly on a subsequence, such that $\bar x_h$ converge to some $\bar x\in[c\wedge d,c\vee d]$), and also using \eqref{ecco}, we get
\[
D_{(-a,a)}(u)\ge\int_{-a}^0\frac{dx}{1+q^2x^2}+\frac 1q\arctan(q\bar x)+\frac 1q\arctan(q(a-\bar x))>\int_{-a}^0\frac{dx}{1+q^2x^2}+\frac 1q\arctan(qa),
\]
since $0<\bar x<a$. The right hand side is exactly $D_{(-a,a)}\left(\wp_{-a;a}^{u(a)}\right)$, this is a contradiction. \EEE
\end{proof}

\begin{proposition}[\bf concave rearrangement]\label{rearr}Let  $a<b$ and let $u$ be a nonincreasing \BBB  absolutely continuous \EEE function on $[a,b]$.
 \EEE Then there exists a nonincreasing concave function $u^\ast\colon[a,b]\to[u(b),u(a)]$ such that 
  $D_{(a,b)}(u)=D_{(a,b)}(u^\ast).$
\end{proposition}
\begin{proof}
Let $\left(u_h\right)_{h\in\mathbb{N}}$ denote a sequence of continuos, piecewise affine, nonincreasing approximating functions, constructed on a equispaced grid of step $(b-a)/h$ on the interval $[a,b]$, and coinciding with $u$ at the nodes of the grid.
\BBB At any differentiability point $x$ of $u$ in $(a,b)$ which for any $h$ is not a grid node (that is, for a.e. $x$ in $(a,b)$), there holds $u_h'(x)\to u'(x)$ as $h\to\infty$. \EEE

For every $h\in\mathbb{N}$ let us exchange the position of each segment of the graph of $u_h$ in such a way that the slopes get ordered in a nonincreasing way.
If $s_{j,h}$ denotes the slope of the piecewise affine function $u_h$ on the interval $[a+(b-a)(j-1)/h, a+(b-a)j/h]$, $j=1,\ldots ,h$, we denote by $s^*_{1,h},\ldots, s^*_{h,h}$ a permutation of the slopes such that $s^*_{1,h}\ge s^*_{2,h}\ge\ldots\ge s^*_{h,h}$. We define $u_h^*$ as the unique continuous, piecewise affine function such that the slope of $u_h^*$ is $s^*_{j,h}$ on the interval $[a+(b-a)(j-1)/h, a+(b-a)j/h]$, $j=1,\ldots ,h$, and such that  $u_h^*(a)=u_h(a)=u(a)$, $u_h^*(b)=u_h(b)=u(b)$. \BBB It is  clear that $D_{(a,b)}(u_h^*)=D_{(a,b)}(u_h)$ for every $h\in\mathbb{N}$.


\BBB  Notice that $(u_h^*)_{h\in\mathbb{N}}$ is a family of concave, uniformly bounded functions on $[a,b]$. By Lemma \ref{compactness} in the Appendix, the family $(u_h^*)_{h\in\mathbb{N}}$ has a concave  decreasing limit point $u^*:[a,b]\to[u(b),u(a)]$ in the strong $W^{1,1}_{loc}((a,b))$ topology (it is extended by continuity to the closed interval). This entails uniform convergence on compact subsets of $(a,b)$ and $a.e.$ convergence of derivatives (up to extracting a subsequence), allowing to pass to the limit with dominated convergence and to get
\begin{equation*}\label{4}
D_{(a,b)}(u^*)=\lim_{h\to\infty}D_{(a,b)}(u_h^*)=\lim_{h\to\infty}D_{(a,b)}(u_h)=D_{(a,b)}(u).
\end{equation*}
Hence, $u^*$ is the desired concave rearrangement.
\end{proof}

\BBB
\begin{remark}\label{plus}\rm
In the same assumptions of {\rm Proposition \ref{rearr}} and with the same notation, if $c<0$ exists such that the set of differentiability points of $u$ with $u'>c$ has positive measure, the same property holds for $u^*$ as well. Indeed, in such case there exists $\eps>0$ such that the set $B$ where $u'>c+\eps$ has positive measure as well. Since $u_h'$ converge to $u'$ a.e. on $B$,  by Egorov theorem there is a positive measure subset $B^*$ of $B$ such that $u'_h\to u'$ uniformly on $B^*$. Then there exists $h_0>0$ such that, for any $h>h_0$ and any $x\in B^*$, there holds $u_h'(x)>c+\eps/2$. For any $h>h_0$, after rearranging, since $u_h^*$ are concave, we have $(u_h^*)'>c+\eps/2$ a.e. on an interval $(a,\xi)$ with length equal to the measure of $B^*$. Since $(u_h^*)'\to (u^*)'$ a.e. on $(a,b)$, we conclude that $(u^*)'\ge c+\eps/2>c$ a.e. on $(a,\xi)$.
\end{remark}
\EEE

For the proof of Lemma \ref{strucside} below, we will need a general result about the resistance functional, holding also in higher dimension. It is the property $|\nabla u|\notin (0,1)$, a proof of which is given  in \cite[Theorem 2.3]{BFK2}. In  dimension one we provide a simpler proof with the following

\begin{proposition}\label{spez}
Let  $a<b$ and let $u$ be a concave, nonincreasing, continuous function on $[a,b]$, \BBB such that $u(a)>u(b)$. \EEE
Then there exists $c\in [a,b)$ such that $u(a)-u(b)\ge b-c$ and
\[D_{(a,b)}(u)\ge D_{(a,b)}(u_{a;b}^{c}),\]
where $u_{a;b}^{c}\colon[a,b]\to \mathbb{R}$ is defined by
\begin{equation}\label{spezzata}u^c_{a;b}(y):=\left\{\begin{array}{lcl}
u(a)&\text{if}&y\in[a,c]\\ \\
\dfrac{y-b}{c-b}\left(u(a)-u(b)\right)+u(b)&\text{if}&y\in(c,b].
\end{array}
\right.\end{equation}
\end{proposition}
\begin{proof} 
Since $u$ is concave,
 then the set $A_u:=\{x\in (a,b)\colon u'_+(x)\ge -1\}$ is connected,  and we define
\[
c^\ast:=
\left\{
\begin{array}{lll}
\sup{A_u}&\text{if}&A_u\neq\varnothing\\
a&\text{if}&A_u=\varnothing
\end{array}
\right.\]
and $u_{c^\ast}\colon [a,b]\to\mathbb{R}$ as follows:
\[u_{c^\ast}(x):=\left\{
\begin{array}{lll}
u(a)&\text{if}&x\in[a,c^\ast+u(c^\ast)-u(a)]\\
-x+c^\ast+u(c^\ast)&\text{if}&x\in [c^\ast+u(c^\ast)-u(a),c^\ast]\\
u(x)&\text{if}&x\in[c^\ast,b].
\end{array}
\right.
\]
Since $\tfrac{1}{1+t^2}\ge 1+\tfrac{t}{2}$ for every $t\le 0$, we have
\[D_{(a,c^\ast)}(u)\ge\int\limits_a^{c^\ast}\left(1+\dfrac{u'(x)}{2}\right)dx=c^\ast-a+\dfrac12(u(c^\ast)-u(a))=D_{(a,c^\ast)}(u_{c^\ast})\] where the last equality follows by a simple calculation. Then $D_{(a,b)}(u)\ge D_{(a,b)}(u_{c^\ast})$.

Let now $c:=c^\ast+u(c^\ast)-u(a)$. We claim that $D_{(a,b)}(u_{c^\ast})\ge D_{(a,b)}(u^c_{a;b})$. To see this, it is enough to prove that $D_{(c,b)}(u_{c^\ast})\ge D_{(c,b)}(u^c_{a;b})$. This immediately follows by Jensen inequality, since the function $f\colon (-\infty,-1]\to\mathbb{R}$ defined by $f(t)=\tfrac{1}{1+t^2}$ is convex and $u_{c^\ast}'\le -1$ a.e. in  $ (c,b)$.
\end{proof}

\begin{corollary}\label{megliospez}
Let  $a<b$ and let $u$ be a nonincreasing absolutely continuous function on $[a,b]$, \BBB such that $u(a)>u(b)$. \EEE Then there exists $c\in [a,b)$ such that $u(a)-u(b)\ge b-c$ and
\begin{equation*}
D_{(a,b)}(u)\ge D_{(a,b)}(u^c_{a;b}),
\end{equation*}
where $u^c_{a;b}\colon [a,b]\to\mathbb{R}$ is the function defined in \eqref{spezzata}.
\end{corollary}
\begin{proof}	\BBB
We apply Proposition \ref{rearr} to $u$, obtaining a nonincreasing concave function $u^\ast\colon[a,b]\to[u(b),u(a)]$ such that
$D_{(a,b)}(u)=D_{(a,b)}(u^\ast)$. Since $u(a)>u(b)$, then $u^*$ is non constant.  We apply Proposition \ref{spez} to $u^*$, obtaining $c\in [a,b)$ and $(u^\ast)^c_{a;b}$,  defined by means of \eqref{spezzata}, such that $u(a)-u(b)\ge u^*(a)-u^*(b)\ge b-c$ and  $D_{(a,b)}(u^*)\ge D_{(a,b)}((u^*)^c_{a;b})$.
But then we easily see that $D_{(a,b)}((u^*)^c_{a;b})\ge D_{(a,b)}(u^c_{a;b})$ and we conclude.
\end{proof}

\BBB
\begin{remark}\rm Notice that the condition $u(a)-u(b)\ge b-c$ on $c$ indicates that the straight line corresponding to the restriction of  $u_{a;b}^c$ on $[c,b]$ has slope smaller than or equal to $-1$.
\end{remark}
\EEE

\begin{lemma}[\bf The side]
\label{strucside} Let
$a<b$ and $q\ge 0$. Let $u$ be a $q$-concave continuous function on $[a,b]$ such that $u(y)\le u(a)$ for every $y\in [a,b]$ and $u(b)<u(a)$. Then there exists $\gamma\in [a,b)$ such that \BBB $u(a)-u(b)\ge b-\gamma$ and \EEE
\begin{equation*}\label{parspez}
D_{(a,b)}(u)\ge D_{(a,b)}(w_{a,\gamma,b}),
\end{equation*}
where $w_{a,\gamma,b}\colon [a,b]\to \mathbb{R}$ is defined by
\begin{equation}\label{ugamma}w_{a,\gamma,b}(y):=\left\{\begin{array}{lcl}
\wp_{a;\gamma}^{u(a)}(y)&\text{if}&y\in[a,\gamma)\\ \\
\dfrac{y-b}{\gamma-b}\left(u(a)-u(b)\right)+u(b)&\text{if}&y\in[\gamma,b].
\end{array}
\right.\end{equation}
The result holds with $\gamma\in (a,b)$ if $u$ is not strictly decreasing on $[a,b]$.
\end{lemma}

\begin{proof}
If $q=0$ we just apply Proposition \ref{spez}, obtaining the concave function $u_{a;b}^c$, defined in \eqref{spezzata}, with $c\in [a,b)$, such that $D_{(a,b)}(u)\ge D_{(a,b)}(u_{a;b}^c) $. Then we just let $\gamma=c$ and observe that in case $q=0$ we have $u_{a;b}^\gamma=w_{a,\gamma,b}$.
If $u$ is not strictly decreasing and  it is concave, then it has a flat part in a neighborhood of $a$ and we can take $c>a$. This is done by fixing $\tilde a>a$ such that $u(\tilde a)=u(a)$ and by applying Proposition \ref{spez} on $[\tilde a,b]$.
 From here on, we let $q>0$.

As did in the proof of Lemma \ref{struccenter}, we prove the result first for  $q$-concave functions $u$ that satisfy the assumptions (i.e. $u(x)\le u(a)$ on $[a,b]$, $u(a)>u(b)$) and are moreover such that $[a,b]\ni x\mapsto u(x)-\tfrac q2(x-a)(x-b)$ is piecewise linear.
This means that $u$ is piecewise parabolic on $[a,b]$, the second derivative of $u$ being equal to $q$ on each of the finitely many pieces.
Moreover, it is clear that $u$ has a finite number of local maximum points on $[a,b]$.

The main part of the proof is the following claim: there is another piecewise parabolic function $\tilde u$ with the same resistance as $u$, such that $\tilde u(a)=u(a)$, $\tilde u(b)=u(b)$, $\tilde u(x)\le \tilde u(a)$ for any $x\in [a,b]$, and moreover there exists $d\in [a,b)$ such that $\tilde u(d)=\tilde u(a)$ and $\tilde u$ is nonincreasing on $[d,b]$. Notice that the claim is directly proved if $u(a)=u(x)$ for each local maximum point $x$ of $u$ on $[a,b]$. Just let $\tilde u=u$ in this case.

 In general, let us consider the subset  of local maxima $x$ such that  $x=b$ or $u(x)>u(y)$ for any $y\in (x,b]$. More precisely, if $\tilde {\mathcal{M}}$ is the set of local maximum points of $u$ on $[a,b]$, we define
\[
\mathcal{M}:=(\tilde {\mathcal{M}}\cap \{b\})\cup \{x\in \tilde{ \mathcal{M}}\colon u(x)>u(y) \text{ for every }y\in(x,b]\}.
\]
Notice that $b$ could be a local maximum point itself, in such case it belongs to $\mathcal{M}$. We also let $x_0:=\min \mathcal{M}$ and $x^*:=\max \mathcal{M}$ (possibly $x_0=a$, $x^*=b$). If $\mathcal{M}$ is reduced to $x_0$, the claim is proved by letting $\tilde u=u$. Otherwise,
for every $x\in  \mathcal{M}\setminus\{x^*\}$ we let
\[\xi_x:=\min\left\{y\in \mathcal{M}\colon y>x\right\},\quad
 z_x:=\min\left\{y\in[x,\xi_x]\colon u(y)=u(\xi_x)\right\}.\]
We let moreover
\[
\gamma_x:=\sum\limits_{{s\in \mathcal{M} },\,{s<x}}(z_s-s),\;\;\mbox {for any $x\in \mathcal{M}$ (notice that $\gamma_{x_0}=0$)},
\]
$$\delta_*:= \sum\limits_{x\in \mathcal{M}\setminus \{x^*\}}(\xi_x-z_x)=x^*-\gamma_{x^*}-x_0.$$
We define $\tilde u:[a,b]\to\mathbb{R}$ by $\tilde u(y)=u(y)$ if $y\in [a,x_0)\cup[x^*,b]$ and, for every $x\in \mathcal{M}\setminus\{x^*\}$,
  \[\tilde u(y)=\left\{\begin{array}{lll}
     u(y+\gamma_x+z_x-x)+u(x_0)-u(\xi_x)&\, &\mbox{if $y\in [x-\gamma_x,x-\gamma_x+\xi_x-z_x)$}\\
  u(y-x_0-\delta_*-\gamma_x+x)&\, &\mbox{if $y\in [x_0+\delta_*+\gamma_x,x_0+\delta_*+\gamma_x+z_x-x)$}.
  \end{array}\right.\]
  Notice that $\tilde u$ is absolutely continuous on $[a,b]$ and that $\tilde u(a+\delta_*)=\tilde u(a)$, moreover $\tilde u$ is nonincreasing on $[a+\delta_*,b]$.
  $\tilde u$ is obtained from $u$ by translating restrictions of $u$ on a finite number of subintervals which cover $[a,b]$.  Then it is piecewise parabolic and
  by the translation invariance property of the resistance functional in  dimension one, we have, for every $x\in \mathcal{M}\setminus\{x^*\}$,
   $$
   D_{(z_x,\xi_x)}(u)=D_{(x-\gamma_x,x-\gamma_x+\xi_x-z_x)}(\tilde u), \quad  D_{(x,z_x)}(u)=D_{(x_0+\delta_*+\gamma_x ,x_0+\delta_*+\gamma_x+z_x-x)}(\tilde u).
   $$
    Therefore $D_{(a,b)}(\tilde u)=D_{(a,b)}(u)$ and the claim is proved, with $d=a+\delta_*$.

  We  apply now Corollary \ref{megliospez} to $\tilde u$ on $[d,b]$, obtaining $\gamma\in [d,b)$ \BBB such that $\gamma\ge b-\tilde u(d)+\tilde u(b)$ and  $D_{(d,b)}(\tilde u)\ge D_{(d,b)}(\tilde u_{d;b}^\gamma)$, where $\tilde u^\gamma_{d;b}$ is defined as \eqref{spezzata}, starting from $\tilde u$.  \EEE Then,  applying Lemma \ref{struccenter} on $[a,\gamma]$, \BBB since $\tilde u(d)=\tilde u(a)=u(a)$ and $\tilde u(b)=u(b)$, we get
   $$D_{(a,b)}(u)=D_{(a,b)}(\tilde u)\ge D_{(a,d)}(\tilde u)+ D_{(d,b)}(\tilde u_{d,b}^\gamma)\ge
     D_{(a,b)}(w_{a,\gamma,b}),$$
with $\gamma\ge b-u(a)+u(b)$ and $\gamma\ge d=a+\delta_*\ge a$. In particular we deduce
\begin{equation}\label{infimum}D_{(a,b)}(u)\ge \inf\left\{D_{(a,b)}(w_{a,\gamma,b})\colon \gamma\in \left[a\vee(u(b)-u(a)+b),b\right)\right\}.\end{equation}

  In order to conclude, we need to prove \eqref{infimum} for a generic $u$ satisfying the assumptions of this lemma. If $u_h$ is a sequence of piecewise parabolic approximations of $u$ constructed by means of Proposition \ref{approx}, we have $u_h(a)=u(a), u_h(b)=u(b)$ and $u_h(x)\le u(x)\le u(a)$ if $x\in [a,b]$, for any $h\in\mathbb{N}$. Therefore we may apply \eqref{infimum}  to $u_h$ and pass it to the limit, since we can use
 Proposition \ref{Dconvergence}, and since  the right hand side of \eqref{infimum} is independent of $h$.
  The map
  $$[a,b)\ni \gamma\mapsto D_{(a,b)}(w_{a,\gamma,b})= \frac{(b-\gamma)^3}{(b-\gamma)^2+(u(a)-u(b))^2}+\frac{2}{q}\arctan{(\tfrac q2(\gamma-a))}$$
  is however smooth and strictly increasing in a left neighborhood of $b$, so that its infimum is realized and belongs to $[a\vee(u(b)-u(a)+b),b)$.
In other words, there is $\gamma\in[a,b)$ such that $\gamma\ge u(b)-u(a)+b$ and  $D_{(a,b)}(u)\ge D_{(a,b)}(w_{a,\gamma,b})$, as desired.

\BBB
Eventually, we prove the last statement, which is in fact obvious if $u(\tilde a)=u(a)$ for some $\tilde a>a$. We assume therefore that $u$ is not strictly decreasing on $[a,b]$ and also that $u(y)<u(a)$ for any $y\in(a,b]$. Then there exists a local maximum point for $u$ in $(a,b]$ that we denote by $a_1$, and we let $\delta_0\in (0,a_1-a)$ be small enough, such that $u(y)\le u(a_1)$ for any $y\in(a_1-\delta_0,a_1)$. We take advantage of Remark \ref{nonuniform} for approximating $u$, by taking a sequence $u_h$ of piecewise parabolic approximations such that $u_h(a_1)=u(a_1)$ for any $h\in\mathbb{N}$. Notice that by construction  $u_h\le u$, thus we have $u_h(y)<u(a)$ for any $y\in (a,b]$, $a_1$ is a local maximum point for $u_h$ and in particular  \begin{equation}\label{spazio}u_h(y)\le u_h(a_1)\quad \mbox{for any}\quad y\in(a_1-\delta_0, a_1),\end{equation} for any $h\in\mathbb{N}$. Now we fix $h$ and for the function $u_h$ we define  $\mathcal{M}$, $x^*$, $x_0$, $d$, $\delta_*$  as above, omitting for simplicity the dependence on $h$. Since $u_h<u(a)$ on $(a,b]$ we readily have  $a=x_0\in\mathcal{M}$. We take the largest element $x$ of $\mathcal{M}$ which is strictly smaller than $a_1$, and since $a_1$ is a local maximum point for $u_h$ (and the rightmost local maximum of $u_h$ necessarily belongs to $\mathcal{M}$), we see that  $x<x^*$, i.e. $x\in\mathcal{M}\setminus\{x^*\}$. Then, by definition of $\xi_x$ above, we get $\xi_x\ge a_1>x$ and $u_h(\xi_x)\ge u_h(a_1)$. Moreover, by the definition of $z_x$ above,  thanks to \eqref{spazio} and to the intermediate value theorem, we get $\xi_x-z_x\ge \delta_0$, implying $\delta_*\ge \delta_0$, i.e., $d\ge a+\delta_0$. Since $\delta_0$ does not depend on $h$, when applying the previous part of this proof we get the improved estimate $D_{(a,b)}(u)\ge \inf\{D_{(a,b)}(w_{a,\gamma,b})\colon \gamma\in [(a+\delta_0)\vee(u(b)-u(a)+b),b)\}$, where the infimum is realized, yielding the result.
\end{proof}
\EEE


\subsection*{Conclusion of the one-dimensional case}

We first combine Lemma \ref{struccenter} and Lemma \ref{strucside} to obtain the following
\begin{proposition}\label{rettaparabolaretta}
Let $M>0$, $q\ge 0$, $u\in \mathcal{K}_q^M$  and $M\ge m:=\max\{u(x)\colon x\in [-1,1]\}$.
Then there exist $\alpha\in [0,m]$, $\beta\in[0,m]$ and $a,b\in\mathbb{R}$, with
\[-1\le a\le \min\{1,-1+m\},\quad\max\{-1,1-m\}\le b\le 1,\quad a\le b,\]  such that the $q$-concave function on $[-1,1]$ defined by
\begin{equation}\label{mista}\hat u(x):= \left\{
\begin{array}{lcl}\vspace{0.15cm}
\dfrac{m-\alpha}{a+1}(x+1)+\alpha&\text{if}&x\in[-1,a)\\\vspace{0.15cm}
\dfrac q2(x-a)(x-b)+m&\text{if}&x\in [a,b]\\
\dfrac{\beta-m}{1-b}(x-b)+m&\text{if}&x\in(b,1]
\end{array}
\right.\end{equation}
satisfies $D_{(-1,1)}(u)\ge D_{(-1,1)}(\hat u)$.
\end{proposition}
\begin{proof}
  We can assume wlog that $u$ is continuous up to the boundary of $[-1,1]$, and we let $\alpha:=u(-1)$ and $\beta:=u(1)$. We take a maximum point $x^*\in[-1,1]$ for $u$. We apply Lemma \ref{strucside} on $[x^*,1]$ and its reflected version on $[-1,x^*]$, finding two points $a,b\in [-1,1]$, with $-1\le a\le x^*\le b\le 1$, such that $D_{(a,b)}(u)\ge D_{(a,b)}(\tilde u)$, where $\tilde u$, by this application of Lemma \ref{strucside}, is made of two straight lines on $[-1,a)$ and $(b,1]$, with slope in modulus greater than or equal to $1$,   and moreover $\tilde u(a)=\tilde u(b)=m$. We change $\tilde u$ with $x\mapsto m+ \tfrac q2(x-a)(x-b)$ on $[a,b]$, and the result follows by means of Lemma \ref{struccenter}.  All the degenerate cases $a=b$, $a=-1$, $b=1$, $a=b=1$, $a=b=-1$ are possible (for instance if $u\equiv m$ on $[-1,1]$, we are just applying Lemma \ref{struccenter}).
\end{proof}
For $M>0$, $q\ge 0$, the resistance of $\hat u$ in 	\eqref{mista} is given explicitly by $D_{(-1,1)}(\hat u)=\Gamma(a,b,m,\alpha,\beta)$, where, if $q>0$,
\begin{equation*}\label{Gamma}\Gamma(a,b,m,\alpha,\beta):=\dfrac{(a+1)^3}{(a+1)^2+(m-\alpha)^2}
+\dfrac2q\arctan\left(\dfrac{q}2(b-a)\right)+\dfrac{(1-b)^3}{(1-b)^2+(\beta-m)^2}\end{equation*}
and where the parameters  $(a,b,m,\alpha,\beta)$ vary in the set
\begin{multline*}\label{T}\mathcal{T}:=\{(a,b,m,\alpha,\beta)\colon
-1\le a\le \min\{1,-1+m\},\, \max\{-1,1-m\}\le b\le 1,\\\, a\le b,\,
0\le m \le M,\,0\le\alpha\le m,\,0\le\beta\le m\}.\end{multline*}
If $q=0$ the $\arctan$ term simply becomes $b-a$.

With the next three propositions we solve the problem $\min_{\mathcal{T}}\Gamma$, for $q\in [0,1]$ and $2M\ge q$.
\begin{proposition}\label{firstreduction}
If $(a,b,m,\alpha,\beta)$ is a minimizer of $\Gamma$ on $\mathcal{T}$, then $\alpha=\beta=0$, $m=M$, $-a=b=:\gamma$ and $\max\{0,1-M\}\le\gamma<1$.
\end{proposition}
\begin{proof}
We first notice that if $(a,b,m,\alpha,\beta)\in\mathcal{T}$ is a point of minimum for $\Gamma$, then both $a\neq-1$ and $b\neq 1$. Since the proofs are similar, let's see, for example, that $a\ne -1$, which is equivalent to show that every $(-1,b,m,\alpha,\beta)\in\mathcal{T}$ is not a point of minimum for $\Gamma$ on $\mathcal{T}$.
Let $\max\{-1,1-m\}\le b\le 1,\,
0\le m \le M,\,0\le\alpha\le m,\,0\le\beta\le m$ be fixed. Then\[\lim\limits_{a\to-1^+}\frac{\partial\Gamma}{\partial a} (a,b,m,\alpha,\beta)=-\dfrac{4}{4+q^2(b+1)^2}<0\] and the thesis is proved for $b\in(-1,1]$. On the other hand it is easily seen that $(-1,-1,m,\alpha,\beta)$ is a local maximum for the function $a\mapsto \Gamma(a,a,m,\alpha,\beta)$, then the proof is done.
So, from now on, we will assume both $a\neq-1$ and $b\neq 1$.

Since the function $m\mapsto \Gamma(a,b,m,\alpha,\beta)$ is decreasing on $[0,M]$, we have that
\begin{equation*}\label{maxx}\Gamma(a,b,m,\alpha,\beta)\ge\Gamma(a,b,M,\alpha,\beta)\end{equation*} for every $(a,b,m,\alpha,\beta)\in\mathcal{T}$, with strict inequality if $m<M$. Moreover since both the functions $\alpha\mapsto \Gamma(a,b,M,\alpha,\beta) $ and $\beta\mapsto \Gamma(a,b,M,0,\beta)$ are non-decreasing on $[0,M]$ we have
\begin{equation*}\label{zeroestr}
\Gamma(a,b,M,\alpha,\beta)\ge\Gamma(a,b,M,0,0),
\end{equation*}
with strict inequality if $\alpha>0$ or $\beta>0$.
Finally, since the function $[0,M]\ni \sigma\mapsto \sigma^3(M^2+\sigma^2)^{-2} $ is convex, and taking into account that both $a+1,1-b\in [0,M]$, the following holds:
\[\Gamma(a,b,M,0,0)\ge \Gamma\left(\frac{a-b}{2},\frac{b-a}{2},M,0,0\right),\]
with strict inequality if $a\neq -b$.
In conclusion, in order to minimize $\Gamma$ on $\mathcal{T}$ we can restrict to $m=M$, $\alpha=\beta=0$, $b=-a=:\gamma\ge0$, $\max\{0,1-M\}\le\gamma<1.$
\end{proof}

\begin{proposition}\label{prel}Let $M>0$ and $q\in[0,1]$ such that $2M\ge q$. Let $\varphi_{M;q}\colon[0,1]\to \mathbb{R}$ be the function defined by \begin{equation}\label{phimq}\varphi_{M;q}(\gamma):=M^4-M^2(1-\gamma)^2-q^2\gamma^2(1-\gamma)^4-3M^2q^2\gamma^2(1-\gamma)^2.\end{equation}
 Then $\varphi_{M;q}$ is strictly increasing  on $[0,1]$.
\end{proposition}
\begin{proof}
If $q=0$ the result is obvious. Assume $q>0$.
We first consider the function $\psi_{M;q}	\colon[0,1]\to\mathbb{R}$ defined by
\[\psi_{M;q}(\gamma):=M^2-q^2\gamma(1-\gamma)^3+2q^2\gamma^2(1-\gamma)^2-3M^2q^2\gamma(1-\gamma)+3M^2q^2\gamma^2,\]
 and we observe that
\begin{equation}\label{factor}\varphi_{M;q}'(\gamma)=2(1-\gamma)\psi_{M;q}(\gamma)\quad \text{for every }\gamma\in[0,1].\end{equation}
 Let now $\alpha,\beta\colon [0,1]\to\mathbb{R}$ be the functions defined by
\[\alpha(\gamma):=-\gamma(1-\gamma)^3\quad\text{and }\quad \beta(\gamma):=2\gamma^2-\gamma.\] It is easy to check that \[\min\limits_{[0,1]}\alpha=\alpha\left(\tfrac14\right)=-\tfrac{27}{256},\quad \min\limits_{[0,1]}\beta=\beta\left(\tfrac14\right)=-\tfrac18.\]
Then, taking into account that $q\in (0,1]$ and $2M\ge q$,  we have
\begin{align*}\psi_{M;q}(\gamma)&=M^2+q^2 \alpha(\gamma)+2q^2\gamma^2(1-\gamma)^2+3M^2q^2\beta(\gamma)
\\&
\ge M^2-\tfrac{27}{256}q^2-\tfrac38M^2q^2
\ge\tfrac58 M^2-\tfrac{27}{256}q^2
\ge\tfrac{13}{64}q^2 >0\end{align*} for every $\gamma\in [0,1].$ Therefore, from \eqref{factor} we conclude.
\end{proof}
\begin{proposition}\label{endprop}
Let $M>0$ and $q\in[0,1]$ such that $2M\ge q$. Let $R_{M;q}\colon[0,1]\to \mathbb{R}$ be the function defined by \eqref{RMQ}.
\begin{itemize}
\item[(i)] If $M\in (0,1)$ then there exists a unique $\gamma^\ast_{M,q}\in (0,1)$ such that  \[\min\limits_{\gamma\in[0,1]}R_{M;q}(\gamma)=R_{M;q}(\gamma^\ast_{M,q}).\]
\item[(ii)]If $M\ge 1$, then $\min\limits_{\gamma\in[0,1]}R_{M;q}(\gamma)=R_{M;q}(0)=\frac{1}{1+M^2}$, and $0$ is the unique minimizer.
\end{itemize}
\end{proposition}
\begin{proof} We first notice that
\[R_{M;q}'(\gamma)=\dfrac{\varphi_{M;q}(\gamma)}{\left(1+q^2\gamma^2\right)[M^2+(1-\gamma)^2]^2}\]
for every $\gamma\in [0,1]$, $\varphi_{M;q}(\gamma)$ being the function defined in \eqref{phimq}. Then the sign of $R_{M;q}'$ coincides with the sign of $\varphi_{M;q}$.

\noindent (i) If $M\in (0,1)$ then $\varphi_{M;q}(0)=M^2(M^2-1)<0$ and $\varphi_{M;q}(1)=M^4>0$. Then, by Proposition  \ref{prel}, there exists a unique $\gamma^{\ast}_{M;q}\in (0,1)$ such that  \[R_{M;q}'(\gamma^{\ast}_{M;q})=\varphi_{M;q}(\gamma^{\ast}_{M;q})=0\] and $R_{M;q}'$ is negative on $[0,\gamma^{\ast}_{M;q})$, while it is positive on $(\gamma^{\ast}_{M;q},1]$. Therefore $\gamma^{\ast}_{M;q}$ is the unique point of minimum of $R_{M;q}$ on [0,1].

\noindent (ii) If $M\ge 1$ then  $\varphi_{M;q}(0)=M^2(M^2-1)\ge 0$ and $\varphi_{M;q}(1)=M^4>0$.  By Proposition \ref{prel}, both $\varphi_{M;q}$ and $R_{M;q}$ are strictly increasing on $[0,1]$, then   \[\min\limits_{[0,1]}R_{M;q}=R_{M;q}(0)=\frac{1}{1+M^2}\]
and $0$ is the unique minimizer of $R_{M;q}$  on $[0,1]$.
\end{proof}

 \noindent\textbf{Proof of Theorem \ref{main1d}}
\begin{proof}
Let $M>0$, $q\in [0,1]$ and $M\ge 2q$.
%
Assume that $u$ is a solution to \eqref{qconcprob}. We may assume that it is not constant and continuous up to the boundary. Let $m\in (0,M]$ be the maximal value of $u$ on $[-1,1]$, and let $$ \xi=\max\{x\in[-1,1]: u(x)=m\},\quad  \eta=\min\{x\in[-1,1]: u(x)=m\}.$$

We claim that $m=M$, $-1<\eta\le\xi<1$ and $u(\pm 1)=0$. If for instance $\eta=-1$ we apply Lemma \ref{strucside} on $[-1,1]$ (reduced to Lemma \ref{struccenter} if $\xi=1$), yielding a competitor of the form of \eqref{mista}. It is not  optimal, as a consequence of Proposition \ref{firstreduction}. This is a contradiction.
Similarly, there holds $\xi<1$. If $m<M$, $u(-1)>0$ or $u(1)> 0$, still we easily have a contradiction by constructing $\hat u$ of the form of \eqref{mista} with $\hat u(\pm1)=u(\pm 1)$, $\max_{[-1,1]} \hat u=m$, and $D_{(-1,1)}(u)\ge D_{(-1,1)}(\hat u)$ (see Proposition \ref{rettaparabolaretta}). But then Proposition \ref{firstreduction} shows that $\hat u$ is non optimal. The claim is proved.

By Lemma \ref{struccenter}, $u$ coincides with $\wp_{\eta;\xi}^M$ on $[\eta,\xi]$, and  the second claim is that $u$ is strictly decreasing on $[\xi,1]$. Indeed, 
if it is not the case we may define $ u_*\in\mathcal{K}_q^M$ by
\begin{equation*}u_*(y):=\left\{\begin{array}{lcl}
u(y)&\text{if}&y\in[-1,\xi)\\
w_{\xi,\zeta,1}(y)&\text{if}&y\in[\xi,1],
\end{array}
\right.\end{equation*}
where $w_{\cdot,\cdot,\cdot}$ is defined in \eqref{ugamma}. Lemma \ref{strucside} shows that $D_{(-1,1)}(u)\ge D_{(-1,1)}(u_*)$ for a suitable  $\zeta\in(\xi,1)$. However we have a contradiction as $u_*$ is not a minimizer, since we can decrease its resistance, in an admissible way, by applying Lemma \ref{struccenter} on $[\eta,\zeta]$. The second claim is proved.

\BBB
 The third claim is that a.e. on $(\xi,1)$ the slope of $u$ is not greater than $-1$. Indeed, suppose by contradiction that there is a positive measure subset of $(\xi,1)$ where $u'>-1$. We apply Proposition \ref{rearr} and Remark \ref{plus} to $u$ on $[\xi,1]$, obtaining a concave function on such interval, with $u'>-1$ a.e. on a subinterval $(\xi,\xi')$, $\xi'>\xi$, and leaving the resistance unchanged. Then we apply Proposition \ref{spez}, obtaining an admissible competitor (up to a vertical translation) with not larger resistance and a flat part on a suitable interval $(\xi,\xi'')$, $\xi''>\xi$. This is a contradiction, because the latter competitor does not have minimal resistance, again its resistance can be improved by applying Lemma \ref{struccenter} on $[\eta, \xi'']$. This proves the third claim.
   \EEE

 The same reasoning applies on $[-1,\eta]$, i.e.  $u$ is strictly increasing on $[-1,\eta]$ with slope a.e. greater than or equal to $1$. The slope of $u$ is in fact constant on $[-1,\eta]$, and on $[\xi,1]$ as well, otherwise Jensen inequality, owing to the strict convexity of the map $t\mapsto \tfrac1{1+t^2}$ for $|t|\ge 1$ would yield a contradiction. For the same reason, as seen in the proof of Proposition \ref{firstreduction}, the two slopes are opposite.

Summing up, if  $u$ is a solution than it has the form of $\hat u$ from \eqref{mista}, with $\alpha=\beta=0, \ m=M,\ a=\eta, \ \xi=b,\ \xi=-\eta=:\gamma$, and $\gamma\in [\max\{0,1-M\},1)$.
 However, minimization among profiles of this particular form reduces to minimize the function $R_{M;q}$, defined in \eqref{RMQ},
on the interval $[\max\{0,1-M\},1).$ But Proposition \ref{endprop} shows that  there is a unique minimizer $\gamma^*$ of $R_{M;q}$ on $[0,1]$, satisfying  in particular $\gamma^*\in[\max\{0,1-M\},1)$,   $\gamma^*=0$ if $M\ge 1$ and $\gamma^*\in (0,1)$ if $M\in(0,1)$.   Notice that $u_{M;q}\in \mathcal{K}_q^M$, thanks to the assumption $M\ge 2q$.
\end{proof}

\section{The radial two-dimensional case}\label{2dsec}
For $0\le a\le b$  and locally absolutely continuous functions $u$ on $(a,b)$, we will use the notation
\begin{equation*}\label{energies}
 \mathscr{D}_{(a,b)}(u):=\int_a^b\frac{r\,dr}{1+(u'(r))^2}\end{equation*}
 {and in case  $a=0$} we shall also write $\mathscr{D}_b(u):=\mathscr{D}_{(0,b)}(u).$

 As for the one-dimensional case, the proof of Theorem \ref{radth} requires several preliminary results, the first of which takes the place of Proposition \ref{changevar}.
\begin{proposition}[\bf Radial parallelogram rule]\label{inrad} Let $q\ge 0$. Let $\alpha,\beta$ be such that $0\le\alpha\le\beta$. Then
\begin{equation*}\int\limits_\alpha^\beta\dfrac{r\,dr}{1+q^2(r-\beta)^2}\ge\int\limits_\alpha^\beta\dfrac{r\,dr}{1+q^2(r-\alpha)^2}\end{equation*}
and if $q>0$ equality holds if and only if $\alpha=\beta$.
\end{proposition}
\begin{proof}Let $q>0$. Let  $\varphi(t):=t\arctan t -\ln (1+t^2)$, $t\in[0,+\infty)$. Since $\varphi(0)=0=\varphi'(0)$ and $\varphi''(t)=2t^2(t^2+1)^{-2}> 0$ for every $t\in(0,+\infty)$ then  $\varphi(t)> 0$ for every $t\in(0,+\infty)$.
Since
\begin{equation*}\int\limits_\alpha^\beta\dfrac{r\,dr}{1+q^2(r-\beta)^2}-\int\limits_\alpha^\beta\dfrac{r\,dr}{1+q^2(r-\alpha)^2}= \frac1{q^2}\varphi(q(\beta-\alpha)),\end{equation*}
 the result follows. If $q=0$ the result is obvious.
\end{proof}
By using Proposition \ref{inrad} in place of Proposition \ref{changevar}, we reason as done in Lemma \ref{struccenter}, and we may prove the corresponding characterization of optimal radial profiles in the center. The proof is actually simplified, thanks to the symmetry assumption. 

\begin{lemma}\label{centerad} Let $q\ge 0, \ a>0,\ H\in\mathbb{R}$.
The minimization problem
\[
\min\left\{\mathscr{D}_{(0,a)}(u):\text{$r\mapsto u(r)-\tfrac{q}{2}r^2$ is concave nonincreasing on $[0,a]$, $u(r)\le u(a)=H$ on $[0,a]$}\right\}
\]
admits the  unique solution $u_*(r):=\frac q2 (r^2-a^2)+H$.
\end{lemma}
\begin{proof}
 If $q=0$ the result is trivial. Let $q>0$.
    Since $r\mapsto u(r)-\tfrac q2 r^2$ is concave nonincreasing we get $u'(r)\le qr$ a.e in $(0,a)$. If $u'\ge 0$ a.e. in $(0,a)$, then
 either $u'(r)=qr$ a.e. in $(0,a)$ or by pointwise estimating the integrand we get $\mathscr{D}_{(0,a)}(u)>\mathscr{D}_{(0,a)}(u_*)$.

 Suppose  that that there are negativity points of the left derivative $u'_-$ on $(0,a)$.
 Since $u$ is $q$-concave, $u'_-$ is upper semicontinuous  on $(0,a)$, therefore the set $I:=\{r\in (0,a):u'_{-}(r)<0\}$ is open, thus a (at most) countable union of (nonempty) disjoint open intervals $(\alpha_j,\beta_j)$. Moreover, if $\beta_j<a$  there holds $u'_-(\beta_j)=0$ (left continuity of $u'_-$).  A direct consequence of $q$-concavity and of the constraint $u(r)\le u(a)$ on $[0,a]$ is that $u'_-(r)\ge \tfrac q2(r-a)$  on $(0,a)$, see Proposition \ref{paral}, therefore if instead  $\beta_j=a$  we still have $\lim_{r\to a^-}u'_-(r)=0$. On the other hand, $q$-concavity yields  $0\ge u'_-(r)\ge q(r-\beta_{j})$  on any interval  $(\alpha_{j},\beta_{j})$.
 Since $u'_-<0$ at some point in $(0,a)$, there is at least one of these intervals $(\alpha_j,\beta_j)$. If there exists an index $j$ such that $\alpha_j>0$, Proposition \ref{inrad} entails
\begin{equation*}\begin{aligned}
\int_I\dfrac{rdr}{1+u'(r)^{2}}&=\sum_j\int_{\alpha_j}^{\beta_j}\dfrac{rdr}{1+u'(r)^{2}}\ge \sum_j\int_{\alpha_j}^{\beta_j}\dfrac{rdr}{1+q^2(r-\beta_j)^2}\\&\ge \sum_j\int_{\alpha_j}^{\beta_j}\dfrac{rdr}{1+q^{2}(r-\alpha_j)^2}> \sum_j\int_{I}\dfrac{rdr}{1+q^{2}r^2}=\int_{I}\dfrac{rdr}{1+q^{2}r^{2}}.
\end{aligned}
\end{equation*}
By taking into account that
$$
\displaystyle\int_{[0,a]\setminus I}\dfrac{rdr}{1+u'(r)^{2}}\ge\int_{[0,a]\setminus I}\dfrac{rdr}{1+q^{2}r^{2}},
$$
we get $\mathscr{D}_{a}(u)>\mathscr{D}_{a}(u_*)$.
The remaining case is $I=(0,\beta)$ for some $\beta\in (0,a]$. If $\beta<a$, $q$-concavity and Proposition \ref{inrad} yield
$$
\mathscr{D}_{a}(u)\ge \int_0^\beta\frac{r\,dr}{1+q^2r^2}+\int_\beta^a\frac{r\,dr}{1+q^2(r-\beta)^2}>\int_0^a\frac{r\,dr }{1+q^2r^2}=\mathscr{D}_{a}(u_*).
$$
If $\beta=a$, we use $0\ge u'(r)\ge \tfrac q2(r-a)$ a.e. on $(0,a)$ and we get
$$
\mathscr{D}_{a}(u)\ge \int_0^a\frac{r\,dr}{1+\tfrac{q^2r^2}{4}}>\int_0^a\frac{r\, dr}{1+q^2r^2}=\mathscr{D}_{a}(u_*),
$$
concluding the proof.
\end{proof}

\begin{lemma} \label{stessaquota}Let $q\ge 0$. Let $0\le\alpha\le\gamma\le\beta$ and  $q(\beta-\gamma)\le2$. Let moreover $u\colon [\alpha,\beta]\to\mathbb{R}$ be an absolutely  continuous function such that
\begin{itemize}
\item [(i)] $u(\gamma)=u(\beta)\ge u(r)$ for any $r\in[\gamma,\beta]$ and the restriction of $u$  on $ [\gamma,\beta]$ is $q$-concave;
\item[(ii)]   $u'(r)\le -1$ a.e. on $ (\alpha,\gamma)$.
\end{itemize}
Then
\[\int\limits_\alpha^\beta\dfrac{r\,dr}{1+u'(r)^2}\ge\int\limits_\alpha^\beta\dfrac{r\,dr}{1+{w_u}'(r)^2},\]
where ${w_u}\colon [\alpha,\beta]\to \mathbb{R}$ is  the absolutely continuous function defined by
\begin{equation*}
w_u(r):=\left\{
\begin{array}{lcl}
u(r+\gamma-\alpha)+u(\alpha)-u(\beta)&\text{if}&r\in[\alpha,\alpha+\beta-\gamma]\\
u(r-\beta+\gamma)&\text{if}&r\in[\alpha+\beta-\gamma,\beta].
\end{array}
\right.
\end{equation*}
\end{lemma}
\begin{proof}
Let $q>0$. It is easily seen, by taking (ii) into account, that
\[
\begin{aligned}
\int\limits_\alpha^\beta\dfrac{r\,dr}{1+{w_u}'(r)^2}&=
\int_\gamma^\beta\frac{(r+\alpha-\gamma)\,dr}{1+u'(r)^2}+\int_\alpha^\gamma\frac{(r+\beta-\gamma)\,dr}{1+u'(r)^2}\\&\le
(\alpha-\gamma)\int\limits_{\gamma}^{\beta}\dfrac{dr}{1+ u'(r)^{2}}+\int_\alpha^\beta\frac{r\,dr}{1+u'(r)^2}+
\frac12(\beta-\gamma)(\gamma-\alpha).
\end{aligned}\]
Since (i) holds,  Lemma \ref{struccenter} entails $D_{(\gamma,\beta)}(u)\ge D_{(\gamma,\beta)}(\wp_{\gamma;\beta}^{u(\gamma)})=\tfrac2q\arctan\left(\tfrac q2(\beta-\gamma)\right)$, so that
\[\begin{aligned}\int\limits_\alpha^\beta\dfrac{r\,dr}{1+{w_u}'(r)^2}-\int\limits_\alpha^\beta\dfrac{r\,dr}{1+{u}'(r)^2}&\le
(\alpha-\gamma)\int\limits_{\gamma}^{\beta}\dfrac{dr}{1+ u'(r)^{2}}+\frac12 (\beta-\gamma)(\gamma-\alpha)\\
&\le(\gamma-\alpha)\left[\frac{ \beta-\gamma}2-\frac2q\arctan\left(\tfrac q2(\beta-\gamma)\right)\right]=\frac{\alpha-\gamma}q \psi\left(\tfrac q2(\beta-\gamma)\right)
\end{aligned}\]
where $\psi(z):=2\arctan z -z.$
Since
$\psi(0)=0$, $\psi'(z)=(1-z^2)(1+z^2)^{-1}\ge 0$ for every $z\in [0,1]$ and $\tfrac q2(\beta-\gamma)\in[0,1]$, the result follows. If $q=0$ the term $\tfrac2q\arctan(\tfrac q2(\beta-\gamma))$ becomes $\beta-\gamma$ and the result follows as well.
%
%
\end{proof}

In the one dimensional case, Proposition \ref{spez} is necessary to show that the slope is greater than or equal to $1$ (in modulus) on the profile side.
This property holds true in the radial two-dimensional case as well, even if we look to the class of nondecreasing radial profiles. It is in fact a   consequence of  \cite[Theorem 5.4]{M} (see also \cite{BFK2}). We give a proof with the following lemma.


\begin{lemma}\label{marc} Let $0\le R_{1}< R_{2}$,  $m_{1}> m_{2}$. 
 Let
\begin{equation*}\mathcal{W}:=\left\{u\in W^{1,1}_{\text{loc}}(R_{1},R_{2})\colon u'\le 0\  \text{a.e. in } (R_{1},R_{2}),\; u(R_{1})=m_{1}> m_{2}=u(R_{2})\right\},\end{equation*}
where the boundary values are understood as limits.
Then $ \mathscr{D}_{(R_1,R_2)}$ 
admits a minimizer on $\mathcal{W}$ which is concave in  $(R_1,R_2)$. If $u_*\in\arg\min_{\mathcal{W}}\mathscr{D}_{(R_1,R_2)}$, then $|u_*'(r)|\not\in (0,1)$ for a.e. $r\in (R_1,R_2)$.
\end{lemma}


\begin{proof}
For $u\in\mathcal{W}$ we define
\begin{equation*}
\tilde f(t):=\left\{\begin{array}{ll} \vspace{5pt} \dfrac{2-t}{2} &\ \text{if}\ 0\le t\le 1\\
\dfrac{1}{1+|t|^{2}}& \ \text{if}\ t\ge 1.
\end{array}\right.
\quad\text{and}\qquad
{\tilde {\mathscr{D}}}_{(R_1,R_2)}(u):=
\int_{R_1}^{R_2}r\tilde f(|u'(r)|)dr.
\end{equation*}
It is readily seen that  $\tilde f$ is convex and that
$\lim_{t\to +\infty}\tfrac{\tilde f(t|z|)}{t}=0$ for any $z\in\mathbb{R}$,
hence $\tilde{\mathscr{D}}_{(R_1,R_2)}$ is sequentially l.s.c. with respect to the $w^{*}-BV_{loc}(R_1,R_2)$ convergence. Moreover if $(u_{n})\subset \mathcal W$ is a minimizing sequence for $\tilde{\mathscr{D}}_{(R_1,R_2)}$, then
$$\int_{R_1}^{R_2}|u'_{n}(r)|\,dr=m_{1}-m_{2},$$
which entails existence of minimizers of ${\tilde {\mathscr{D}}}_{(R,1)}$ on $\mathcal{W}$.
Let now $R_1\le \alpha < \gamma\le \beta\le R_2$ and let $w\in \mathcal W$ be a piecewise affine function with slopes $\xi_{1}\le 0$ in $(\alpha,\gamma)$ and $\xi_2\le 0$ in $(\gamma,\beta)$, such that $\xi_1\le\xi_2$. Then, by setting  $\lambda :=(\gamma-\alpha)(\beta-\alpha)^{-1}$, we have
\begin{equation*}
\int_{\alpha}^{\beta}r\tilde f(|w'(r)|)\,dr=\dfrac{1}{2}\left ((\gamma^{2}-\alpha^{2})\tilde f(|\xi_{1}|)+(\beta^{2}-\gamma^{2})\tilde f(|\xi_{2}|)\right )
\end{equation*}
 and convexity of $\tilde f(|\cdot|)$ on $(-\infty, 0]$ entails
\begin{equation*}
\int_{\alpha}^{\beta}r\tilde f(|\lambda \xi_{1}+(1-\lambda) \xi_{2}|)\,dr\le \dfrac{1}{2}(\beta^{2}-\alpha^{2})\left (\lambda\tilde f(|\xi_{1}|)+(1-\lambda)\tilde f(|\xi_{2}|)\right ).
\end{equation*}
By taking into account that $\tilde f$ is decreasing we get
\begin{equation}\label{concenv}
\int_{\alpha}^{\beta}r\tilde f(|\lambda \xi_{1}+(1-\lambda) \xi_{2}|)\,dr-\int_{\alpha}^{\beta}r\tilde f(|w'|)\,dr\le \dfrac{1}{2}(\beta-\gamma)(\gamma-\alpha)(\tilde f(|\xi_{1}|)-\tilde f(|\xi_{2}|))\le 0.
\end{equation}
Hence, if $w_{**}$ denotes the concave envelope of $w$, \eqref{concenv} entails
$
{\tilde {\mathscr{D}}}_{(R_1,R_2)}(w)-{\tilde {\mathscr{D}}}_{(R_1,R_2)}(w_{**})\ge 0
$
for every  piecewise affine  $w\in \mathcal W$ and therefore for  every $w\in \mathcal W$,
and we may conclude that $\tilde{\mathscr{D}}_{(R_1,R_2)}$ admits a minimizer on $\mathcal{W}_{**}$ and that
\begin{equation}\label{**}\min_{\mathcal W}{\tilde {\mathscr{D}}}_{(R_1,R_2)}=\min_{\mathcal W_{**}}{\tilde {\mathscr{D}}}_{(R_1,R_2)},
\end{equation}
where ${\mathcal W_{**}}:=\{u\in \mathcal W: \ u \text{ is concave}\}$.

Next, we let $u\in\mathcal{W}_{**}$
\EEE and we argue as in \cite[Theorem 2.3]{BFK2}. We let $\bar r:=\inf A_u$, where
$A_u:=\left\{r\in (R_1,R_2): u'_{+}(r)\le -1\right\}\cup\{R_2\}$, \EEE
and
\begin{equation*}
v(r)=\left\{\begin{array}{ll} \min\{u(\bar r)+u'_{+}(\bar r)(r-\bar r),\ m_1\}& \ \text{if}\ r\in (R_1,\bar r)\\
u(r)&\  \text{if}\ r\in [\bar r,R_2).
\end{array}\right.
\end{equation*}
We have $v\in\mathcal{W}_{**}$, $v\ge u$ on $(R_1,R_2)$ and $|v'|\notin (0,1)$ a.e. on $(R_1,R_2)$.  Moreover, $|v'|\in \{0,1\}$ and $|u'|\in (0,1)$ a.e. on the set $I:=\{r\in (R_1,R_2):u(r)\neq v(r)\}$, while $u'=v'$ a.e. on the set $E:=\{r\in(R_1,R_2):u(r)=v(r)\}$. These information on $u',v'$, together with the definition of $\tilde {\mathscr{D}}_{(R_1,R_2)}$, directly entail  $\mathscr{D}_{(R_1,R_2)}(u)\ge \tilde{\mathscr{D}}_{(R_1,R_2)}(u)$, $\mathscr{D}_{(R_1,R_2)}(v)= \tilde{\mathscr{D}}_{(R_1,R_2)}(v)$ and \EEE
\[\begin{aligned}
&\mathscr{D}_{(R_1,R_2)}(u)\ge\tilde{\mathscr{D}}_{(R_1,R_2)}(u)=\int_E r\tilde f(|u'(r)|)\,dr+\int_I r\tilde f(|u'(r)|)\,dr\\\qquad&=\tilde{\mathscr{D}}_{(R_1,R_2)}(v)+\int_I \left(\tilde f(|u'(r)|)-\tilde f(|v'(r)|)\right)r\,dr=\tilde{\mathscr{D}}_{(R_1,R_2)}(v)+\int_{R_1}^{R_2}\frac{|v'(r)|-|u'(r)|}2\,r\,dr
\\\qquad&=\tilde{\mathscr{D}}_{(R_1,R_2)}(v)+\int_{m_2}^{m_1}\frac{v^{-1}(t)-u^{-1}(t)}2\,dt=\mathscr{D}_{(R_1,R_2)}(v)+\int_{m_2}^{m_1}\frac{v^{-1}(t)-u^{-1}(t)}2\,dt,
\end{aligned}
\]
where we changed variables in the last but one equality, taking into account that $u,v$ are concave nonincreasing on $(R_1,R_2)$.
Since $v\ge u$, we conclude that $\mathscr{D}_{(R_1,R_2)}(u)\ge \mathscr{D}_{(R_1,R_2)}(v)$ with equality if and only if $u=v$,  and that the same holds for $\tilde{\mathscr{D}}_{(R_1,R_2)}$. In particular, if $u\in\arg\min_{\mathcal{W}_{**}}\tilde{\mathscr{D}}_{(R_1,R_2)}$, then $u=v$ implying  $\mathscr{D}_{(R_1,R_2)}(u)=\tilde{\mathscr{D}}_{(R_1,R_2)}(u)$, and this entails that  $u$ minimizes also $\mathscr{D}_{(R_1,R_2)}$ on $\mathcal{W}_{**}$. \EEE
  Summing up, $\mathscr{D}_{(R_1,R_2)}$ admits a minimizer on $\mathcal{W}_{**}$, and moreover $u$ is a minimizer of $\tilde{\mathscr{D}}_{(R_1,R_2)}$ on $\mathcal{W}_{**}$ if and only if it is a minimizer of $\mathscr{D}_{(R_1,R_2)}$ on $\mathcal{W}_{**}$, with same  minimal values. In such case  $|u'|\notin (0,1)$ a.e. in $(R_1,R_2)$.

  Let us assume from now on that $u$ is in fact a minimizer of $\mathscr{D}_{(R_1,R_2)}$ on $\mathcal{W}_{**}$.
  If we also take \eqref{**} into account, for every $w\in {\mathcal W}$ we have
\begin{equation}\label{si}\mathscr{D}_{(R_1,R_2)}(u)= \tilde{\mathscr{D}}_{(R_1,R_2)}(u)=\min_{\mathcal{W}_{**}}\tilde{\mathscr{D}}_{(R_1,R_2)}=\min _{\mathcal W}\tilde{\mathscr{D}}_{(R_1,R_2)}\le \tilde{\mathscr{D}}_{(R_1,R_2)}(w)\le \mathscr{D}_{(R_1,R_2)}(w),
\end{equation}
so that $u$ is also a minimizer of $\tilde{\mathscr{D}}_{(R_1,R_2)}$ and of $\mathscr{D}_{(R_1,R_2)}$ on $\mathcal{W}$. It is the desired concave minimizer.
 Eventually,
let $u_*\in\arg\min_{\mathcal{W}}\mathscr{D}_{(R_1,R_2)}$. By definition of $\tilde{\mathscr{D}}_{(R_1,R_2)}$ we have ${\mathscr{D}}_{(R_1,R_2)}(u_*)\ge\tilde{\mathscr{D}}_{(R_1,R_2)}(u_*)$.
We prove that the latter is in fact an equality. Indeed,
   if this was not the case we would be lead, since we have just proved that  ${\mathscr{D}}_{(R_1,R_2)}(u_*)={\mathscr{D}}_{(R_1,R_2)}(u)$, and also using the first equality in \eqref{si}, to $\tilde {\mathscr{D}}_{(R_1,R_2)}(u)={\mathscr{D}}_{(R_1,R_2)}(u)= {\mathscr{D}}_{(R_1,R_2)}(u_*)>\tilde {\mathscr{D}}_{(R_1,R_2)}(u_*)$, against the minimimality of $u$ for $\tilde {\mathscr{D}}_{(R_1,R_2)}$ on $\mathcal{W}$. We conclude that $ {\mathscr{D}}_{(R_1,R_2)}(u_*)=\tilde {\mathscr{D}}_{(R_1,R_2)}(u_*)$, which directly entails $|u_*'|\notin (0,1)$ a.e. in $(R_1,R_2)$.
\end{proof}

The next lemma shows some important properties of solutions of \eqref{radialproblem}.

\begin{lemma}\label{lastbutone} Let $R>0$, $M>0$, $0\le qR\le1$ and $2M\ge qR^2$.
Let $u\in C^0([0,R])$ be  a solution to problem \eqref{radialproblem}. Let $m=\max\{u(x): \ x\in [0,R]\}$  and $ a:=\max\{ x\in  [0,R]: \ u(x)=m\}$.
Then $a< R$, $u$ is strictly decreasing in $[a,R]$ where $u'\le -1$ a.e., $u(R)=0$ and   $m=M$.
\end{lemma}
\begin{proof} 
If $a=R$, by Lemma \ref{centerad} we get that the resistance of $r\mapsto\tfrac{q}{2}(r^{2}-R^{2})+m$ is less than or equal to
$\mathscr{D}_{R}(u)$,
but then it is readily seen that by taking
\begin{equation*}w(r)=\left\{\begin{array}{lcl}\vspace{0.1cm}
\tfrac{q}{2}(r^{2}-(R-\delta)^{2})+m&\text{if}&r\in[0, R-\delta]\\
\frac m\delta(R-r)&\text{if}&r\in(R-\delta ,R]
\end{array}
\right.\end{equation*}
we get $\mathscr{D}_{R}(w+M-m)=\mathscr{D}_{R}(w)< \mathscr{D}_{R}(u)$ for $\delta$ small enough. This contradicts the minimality of $u$, since $r\mapsto w(r)+M-m$ belongs to $\mathcal{R}_{R;M;q}$, as a direct consequence of the high profile assumption $2M\ge qR^2$. We have obtained $a<R$ and $u(R)<m$.

Next we prove that $u$ is strictly decreasing on $[a,R]$.
Notice  that the restriction of $u$ to $[a,R]$ satisfies the assumptions of Lemma \ref{strucside} (here we have $R$ in place of $b$).
  If $u$ is not strictly decreasing in  $[a, R]$, as done in the proof of Lemma \ref{strucside} we fix a local maximum point $a_1\in(a,R]$ of $u$ and we fix $\delta_0>0$  small enough such that $u(r)\le u(a_1)$ for any $r\in (a_1-\delta_0,a_1)$. \EEE
\EEE By means of Remark \ref{nonuniform}, we let $(u_h)_{h\in\mathbb{N}}$ be an approximating sequence of uniformly converging, piecewise parabolic functions on $[0,R]$, such that   $u_h(0)=u(0)$, $u_h(R)=u(R)$, $u_h(a)=u(a)$ and $u_h(a_1)=u(a_1)$ \EEE for every $h\in\mathbb{N}$.
  Of course Proposition \ref{Dconvergence} applies to functional $\mathscr{D}_{R}$ as well, so that  \EEE
\begin{equation}\label{acca1}
\mathscr{D}_{R}(u)\ge \mathscr{D}_{R}(u_h)
-\frac1h, \quad \forall h\in\mathbb{N}.\end{equation}
The argument is similar to the one of Lemma \ref{strucside}, so we shall skip some details. Following the  the proof of Lemma \ref{strucside} , we define the quantities $\tilde {\mathcal{M}}$, $\mathcal{M}$, $x_0$, $x^*$, $\xi_x$, $z_x$, $\delta_*$ for $u_h$, so that they all depend on $h$, even if for simplicity we omit this dependence in the notation. Here we also define $z_{x^*}:=R$ (and it is possible that $x^*=z_{x^*}=R$).   But since $u_h(r)\le u(r)\le u(a_1)$ for any $r\in (a_1-\delta_0,a_1)$, the argument at the end of the proof of Lemma \ref{strucside} shows that $\delta_*\ge \delta_0$ for any $h\in\mathbb{N}$. 
 \EEE On each interval $[x,z_x]$, $x\in \mathcal{M}$, we have that $u_h$ is a strictly decreasing function, as seen in the proof of Lemma \ref{strucside}. We define $\tilde u_h:[0,R]\to\mathbb{R}$ by modifying $u_h$ on each of these intervals. Indeed, by Lemma \ref{marc} we change $u_h$ on $(x,z_x)$, for any $x\in\mathcal{M}$, with a  resistance minimizer (among nonincreasing functions with fixed boundary values) having a flat part on a subinterval $(x,\tilde x)$ and a concave part with slope not greater than $-1$ a.e. on  $(\tilde x,z_x)$, for a suitable $\tilde x\in [x,z_x)$. In this way, we find  $\tilde u_h(x)=u_h(x)$, $\tilde u_h(z_x)=u_h(z_x)$ and $\mathscr{D}_{R}(u_h)\ge \mathscr{D}_{R}(\tilde u_h)$.   Notice that by its definition, the restriction of $\tilde u_h$ on $[a,R]$ is  absolutely continuous. Notice moreover that $[a,R]$ is now partitioned in a finite number of intervals: we have the intervals of the form $[\tilde x,z_x]$, $x\in \mathcal{M}$, where $\tilde u_h$ is concave nonincreasing with slope a.e. not in $(-1,0)$, while in each of the remaining intervals $\tilde u_h$   is $q$-concave with same value at the two endpoints (and by definition of $\delta_*$, if the sum of the lengths of these remaining intervals is $\delta_{**}$, then $\delta_{**}\ge \delta_*$). \EEE  Starting from $\tilde u_h$, by repeatedly applying  Lemma \ref{stessaquota} (notice that this is possible because of the assumption $qR\le1$) we  construct $u^*_h:[0,R]\to\mathbb{R}$ with the following properties: $u_h^*\le m$, $u_h^*\equiv u_h$ on $[0,a]$, $u_h^*$ is $q$-concave on  $[0,a+\delta_{**}]$,  $u_h^*(a)=u_h^*(a+\delta_{**})=m$,  $u_h^*$ is strictly decreasing on $[a+\delta_{**},R]$, \EEE $u_h^*(R)=u_h(R)=u(R)$,  the range of $u_h^*$ is contained in that of $u_h$ 	\EEE and
 \begin{equation}\label{uaccastar}
 \mathscr{D}_{R}(u_h)\ge \mathscr{D}_{R}(u_h^*).
 \end{equation}
  A last application of Lemma \ref{marc} on $[a+\delta_{**},R]$ entails  $\bar u_h$, given by $u_h^*$ on $[0,a+\delta_{**}]$ and by a  concave \EEE resistance minimizer among nonincreasing functions on the interval $[a+\delta_{**},R]$ with fixed values $m$ and $u(R)$ at the two endpoints.  $\bar u_h$ is $q$-concave on the whole $[0,R]$  with $\bar u_h(a)=\bar u_h(a+\delta_{**})=m$, $\bar u_h(R)=u(R)$ and,
   from \eqref{acca1}, \eqref{uaccastar} and Lemma \ref{marc}, it satisfies
 \begin{equation}\label{acca2}
 \mathscr{D}_{R}(u)\ge
  \mathscr{D}_{R}(\bar u_h)-\frac1h, \quad \forall h\in\mathbb{N}.\end{equation}
 As already observed, $\delta_*$ and $\delta_{**}$ might depend on $h$, but $\delta_{**}\ge \delta_*\ge \delta_0$ and the quantity $\delta_0>0$ is fixed and does not depend on $h$.
 $(\bar u_h)_{h\in\mathbb{N}}$ is a sequence of uniformly bounded $q$-concave functions on $[0,R]$ (in particular, the range of $\bar u_h$ is contained in that of $u_h$, which goes to that of $u$ as $h\to \infty$ by uniform convergence). Therefore, we may invoke Lemma \ref{compactness} in the Appendix: up to extraction of a subsequence, $\bar u_h$  converge uniformly on compact subsets of $(0,R)$ (even of $[0,R)$ in this case since $(\bar u_h)'_+(0)\le 0$)  to some $q$-concave function $\bar u:[0,R]\to [0,m]$  (continuous up to redefinition at  $R$), which is moreover satisfying  $\bar u(a)=\bar u(a+\tilde\delta)=m$, for a suitable $\tilde\delta\in[\delta_0, R-a]$. Indeed, we may pass to the limit in the relations $\bar u_h(a)=\bar u_h(a+\delta_{**})=m$, where $\delta_{**}$ depends in general on $h$ and here $\tilde\delta$ is a corresponding limit point.    \EEE From Lemma 	\ref{compactness}  we also have a.e. convergence of derivatives, implying $\mathscr{D}_{{R}}(\bar u_h)\to \mathscr{D}_{{R}}(\bar u)$ as $h\to\infty$.  Together with \eqref{acca2}, this implies $\mathscr{D}_{R}(u)\ge
  \mathscr{D}_{R}(\bar u)$. But now we define $\bar w:[0,R]\to\mathbb{R}$ as
  \begin{equation*}\bar w(r)=\left\{\begin{array}{lcl}\vspace{0.1cm}
m+\tfrac q2 (r^2-(a+\tilde \delta)^2)&\text{if}&r\in[0, a+\tilde\delta]\\
\bar u(r) 	&\text{if}&r\in(a+\tilde\delta ,R],
\end{array}
\right.\end{equation*}
  and since $\tilde\delta>0$ by Lemma \ref{centerad} we find that $\mathscr{D}_{R}(\bar w+M-m)=
  \mathscr{D}_{R}(\bar w)<\mathscr{D}_{R}(\bar u)$, and $r\mapsto \bar w(r)+M-m$ belongs to $\mathcal{R}_{R;M;q}$, since $2M\ge qR^2$, thus contradicting minimality of $u$.

Now we show   that   $u'\le -1$ a.e. in $(a, R)$.
Being the restriction of $u$ to $[a,R]$  nonincreasing, it  necessarily minimizes  the resistance functional  among all nonincreasing $v$ in $[a, R]$ such that $v(a)=m$ and $v(R)=u(R)$,  otherwise the concave minimizer provided by  Lemma \ref{marc} would give a contradiction. As $u<u(a)$ on $(a,R]$, still by Lemma \ref{marc} we get  that   $u'\le -1$ a.e. in $(a, R)$.

\EEE
 If $m<M$ or $u(R)>0$, we let
\begin{equation*}w_*(r)=\left\{\begin{array}{lcl}\vspace{0.1cm}
\tfrac{q}{2}(r^{2}-a^{2})+M-m&\text{if}&r\in[0, a]\\
\tfrac M{m-u(R)} (u(r)-u(R))&\text{if}&r\in(a ,R].
\end{array}
\right.\end{equation*}
Since $u(a)=m$ and $u'\le -1$ on $(a, R)$, it is clear that $w_*\in\mathcal{R}_{R;M;q}$ and that $$\int_a^R\frac {r\,dr}{1+w_*'(r)^2}<\int_a^R \frac{r\,dr}{1+u'(r)^2},$$
and then Lemma \ref{centerad} implies $\mathscr{D}_{R}(w_*)<\mathscr{D}_{R}(u)$, again contradicting minimality of $u$.
\end{proof}

All the necessary elements for the proof of Theorem \ref{radth} are now settled. Before proceeding with the proof, we give a couple of useful result for the analytic characterization of the side of the optimal profile.

\begin{proposition}\label{technical}
Let $M>0, \ R>0,$ and $h:(-\infty-1]\to\mathbb{R}$ be defined by $h(t)=-t(1+t^{2})^{-2}$. Then
\[a_{M}:=\min\left\{ a\in (0,R): \;-\int _{a}^{R}h^{-1}\left(\frac{a}{4r}\right)\,dr\le M\right\}\]
is well defined and it uniquely realizes equality in the above inequality among values in $(0,R)$. Besides, there exists a unique strictly decreasing $C^1$ function $\eta:[a_M,R)\to\mathbb{R}$  such that $0< \eta(a)\le \tfrac{a}{4}$ and
\begin{equation}\label{etaa} -\int _{a}^{R}h^{-1}\left(\tfrac{\eta(a)}{r}\right)\,dr= M
\end{equation}
for every $a\in [a_{M},R)$. Moreover, there holds
\begin{equation}\label{etaprime}
\displaystyle\eta'(a)\int_{a}^{R}\dfrac{dr}{rh'\left(h^{-1}\left(\eta(a)/r\right)\right)}=h^{-1}\left(\tfrac{\eta(a)}a\right).
\end{equation}
\end{proposition}
\begin{proof} Notice that the inverse function $h^{-1}$ is defined on $(0,\tfrac14]$, it is smooth, increasing and there hold $\lim_{r\to 0}h^{-1}(r)=-\infty$ and $h^{-1}(\tfrac14)=-1$.
Let \begin{equation}\label{varphidia}\varphi(a):=-\int _{a}^{R}h^{-1}\left(\frac{a}{4r}\right)\,dr,\quad a\in (0,R).\end{equation}
It is readily seen, from the definition of $h$, that $\lim_{a\to R}\varphi(a)=0,\ \ \lim_{a\to 0}\varphi(a)=+\infty$ and $\varphi'<0$ on  $(0,R)$. Then there exists a unique $a_{M}$ such that $\varphi(a_{M})=M$ and $
[a_{M},R)=\{a\in (0,R): \varphi(a)\le M\}$. For every $a\in [a_{M},R)$
let $\psi_a: (0, \tfrac{a}{4}]\to [0,+\infty)$ be defined by
\[\psi_{a}(\eta):=-\int _{a}^{R}h^{-1}\left(\frac{\eta}{r}\right)\,dr.\]
Similarly as above we may check that for any $a\in[a_M,R)$ there is  \[\psi_{a}'(\eta)=\int _{R}^{a}\dfrac{dr}{rh'(h^{-1}({\eta}/{r} ))}< 0\]
on $(0,\tfrac a4)$, and moreover $\lim_{\eta\to 0}\psi_{a}(\eta)=+\infty,\ \  \lim_{\eta\to a/4}=\varphi(a)\le M$. Hence for every $a\in [a_{M}, R)$ there exists a unique
$\eta\in (0,a/4]$ such that $\psi_a(\eta)=M$ is satisfied, and we denote it by $\eta(a)$.  Notice that $\psi_a(\eta)$ strictly decreases with $a$ for each $\eta\in(0,\tfrac a4]$ so that the function $[a_M,R)\ni a\mapsto \eta(a)$ is strictly decreasing,  and it satisfies \eqref{etaa}. Moreover, we have $\eta(a_M)=\tfrac a4$, $\lim_{a\to R}\eta(a)=0$. $\eta(a)$ is  $C^1$ and satisfies \eqref{etaprime} by the implicit function theorem.
\end{proof}
 \begin{proposition}\label{crucialpt} Let $q\ge 0,\ R>0$, $M>0$  and let $\gamma_q:(0,R)\to\mathbb{R}$ be defined by
  \begin{equation*}\gamma_{q}(a):=\sqrt{\dfrac{1}{2}\big( 3a^{2}q^{2}+1+\sqrt{9a^{4}q^{4}+10a^{2}q^{2}+1}\big)}.\end{equation*}
  Let $h$, $a_M$ be defined as in {\rm Proposition \ref{technical}}.
 Let the function $\zeta_q:(0,R)\to\mathbb{R}$ be  defined by
 \begin{equation*}\zeta_{q}(a):=-\int_{a}^{R}h^{-1}\left(\frac{a h(-\gamma_{q}(a))} {r}\right)\,dr.
\end{equation*}
Then  there exists  a unique $a_{*}\in [a_{M},R)$ such that $\zeta_q(a_*)=M$.
\end{proposition}
\begin{proof}
Notice that
$\zeta_q$ is well defined on $(0,R)$, since $h\le\tfrac 14$ and $\gamma_q\ge 1$.
If $q=0$, then $\gamma_0\equiv 1$, and since $h(-1)=\tfrac14$ we obtain $\zeta_0(a)=\varphi(a)$, where $\varphi$ is defined by \eqref{varphidia}. Therefore, we are reduced to Lemma \ref{technical} in this case, and we find $a_*=a_M$.

Let $q>0$. Then $h(-\gamma_q(a))<\tfrac 14$ on $(0,R)$, so that $-h^{-1}(\tfrac{a_M h(-\gamma_q(a_M))}{r})>-h^{-1}(\tfrac{a_M}{4r})$ on $(a_M,R)$,
 hence, by Lemma \ref{technical}, $\zeta_{q}(a_{M})> M$. On the other hand, $\lim_{a\to R}\zeta_{q}(a)=0$,
and by taking into account that \[\zeta'_{q}(a)=-\gamma_{q}(a)-\int _{a}^{R}\dfrac{dr}{rh'(h^{-1}({ah(-\gamma_{q}(a))}/{r} ))}< 0,\]
the result follows.
\end{proof}

\noindent{\bf Proof of Theorem \ref{radth}}.
\begin{proof}
Let $u\in C^0([0,R])$ be solution to \eqref{radialproblem}.
Since the assumptions of Lemma \ref{lastbutone} are satisfied, we have $u(R)=0$, $\max u=M$, $a:=\max\{x\in[0,R]\colon u(x)=M\}<R$, and moreover $u'\le -1$ on $(a,R)$.
We concentrate on the interval $(a,R)$, where first variation of the resistance functional yields
\[\int_a^{R} \dfrac{ru'\varphi'dr}{(1+u'^{2})^{2}}=0\]
for every $\varphi\in C^1_{0}(a,R)$, that is there exists a constant $\eta>0$ such that
\begin{equation*}\dfrac{-ru'}{(1+u'^{2})^{2}}=\eta\end{equation*}
a.e. in $(a,R)$.
We get therefore
  $h(u'(r))=\eta/r$, $h$ being defined in Proposition \ref{technical}. Hence, $4\eta/r\in (0,1]$ for every $r\in (a,R)$, that is   $0< \eta\le a/4$. Since $u(R)=0,  \ u(a)=M$, then $\eta$ has to satisfy
 \begin{equation*}-\int_{a}^{R}h^{-1}\left(\frac\eta r \right)\,dr=M,
 \end{equation*}
 which implies
 \[-\int_{a}^{R}h^{-1}\left(\frac a{4r}\right)\,dr\le M,\]
 that is $a\in [a_{M}, R)$, where $a_M$ is defined in Proposition \ref{technical}.

 Summing up if $u\in C^{0}([0,R])$ solves \eqref{radialproblem},  there exist  $a\in [a_{M},R)$ and, by Proposition \ref{technical}, a unique $\eta=\eta(a) \in (0, a/4]$
 such that (also using Lemma \ref{centerad}),
 \[u(r)=\dfrac{q}{2}(r^{2}-a^{2})+M
 \quad \text{in $[0,a]$},\]
 \[u(r)=-\int_{r}^{R}h^{-1}\left(\frac{\eta(a)}{s}\right)\,ds
 \quad \text{in $(a,R]$}\] and the latter profile has resistance is given by
 \begin{equation*}
 \mathcal{E}(a):=\int_{0}^{a}\dfrac{r\,dr}{1+q^2r^2}+\int_{a}^{R}\dfrac{r\,dr}{1+|h^{-1}(\eta(a)/{r})|^{2}}.
  \end{equation*}

 We are now left to minimize over $a\in [a_M,R)$. That is,
 we have  $\mathscr{D}_{{{R}}}(u)=\min_{a\in [a_{M},R)}\mathcal E(a)$.
 Proposition \ref{technical}  shows that the map $[a_M,R)\ni a\mapsto \eta(a)$ is $C^1$ and strictly decreasing.
   By using the definition of function $h$, and by taking into account  formula \eqref{etaprime} of Proposition \ref{technical}, we have
  \begin{equation*}\begin{aligned}
 \mathcal E'(a)&=\dfrac{a}{1+q^{2}a^{2}}-\dfrac{a}{1+|h^{-1}(\eta(a)/{a})|^{2}}+2\eta'(a)\int_{a}^{R}\dfrac{-h^{-1}(\eta(a)/{r})\,dr}{(1+|h^{-1}(\eta(a)/{r})|^{2})^{2}h'(h^{-1}(\eta(a)/r))}
  \\&=
  \dfrac{a}{1+q^{2}a^{2}}-\dfrac{a}{1+|h^{-1}(\eta(a)/{a})|^{2}}+2\eta'(a)\eta(a)\int_{a}^{R}\dfrac{dr}{rh'(h^{-1}(\eta(a)/r))}\\
  &=
  \dfrac{a}{1+q^{2}a^{2}}-\dfrac{a}{1+|h^{-1}(\eta(a)/{a})|^{2}}+2\eta(a)h^{-1}(\eta(a)/{a}).
  \end{aligned}
  \end{equation*}
  A computation then shows that $\mathcal E'(a)\ge 0$ if and only if
  \[(1+|h^{-1}(\eta(a)/{a})|^{2})^{2}\ge (3|h^{-1}(\eta(a)/{a})|^{2}+1)(1+q^{2}a^{2})\]
  that is if and only if $h^{-1}(\eta(a)/{a})\le -\gamma_{q}(a)$, where $\gamma_q$ is the function defined in Proposition \ref{crucialpt}, or equivalently $\eta(a)\le ah(-\gamma_{q}(a))$.
  But $\eta(a_{M})=\tfrac{a_{M}}{4}> a_{M}h(-\gamma_{q}(a_{M}))$ while $Rh(-\gamma_{q}(R))>0=\lim_{a\to R}\eta(a)$, hence
  the equation $\eta(a)= ah(-\gamma_{q}(a))$ (equivalent to $\mathcal{E}'(a)=0$) has at least a solution  $a_{*}\in [a_{M}, R)$ which is necessarily unique by Proposition \ref{crucialpt}
  since
  \[-\int_{a_*}^{R}h^{-1}\left(\frac{\eta(a_{*})}{r}\right)\,dr=M=-\int_{a_*}^{R}h^{-1}\left(\frac{a_{*}h(-\gamma_{q}(a_{*}))}{r}\right)\,dr.\]
  Therefore, under the assumptions $0\le qR\le 1$ and $2M\ge qR^2$, problem \eqref{radialproblem} has a unique solution, characterized by the number $a^*$ coming from Proposition \ref{crucialpt}, with  $u'(r)=h^{-1}(\tfrac{\eta(a_{*})}{r})$ in $(a_{*},R)$ and $u(a_*)=M$. The proof is completed.
  \end{proof}

 \begin{remark} 	\rm We note that $\gamma_{0}(a)\equiv 1$, hence when $q=0$ we get $a_{*}=a_{M}$ and $\eta(a_{*})=\tfrac{a_{M}}{4}$,
 thus obtaining the classical concave radial minimizer.
 \end{remark}


\section{Approximation of optimal profiles in the general two-dimensional case}\label{numericalsec}

{

To conclude our study, we discuss the approximation of optimal
$q$-concave graphs with no radiality assumption.
 For $M>0$ and $q>0$, we provide in this section a numerical optimization algorithm
to approximate   $q$-concave profiles  of  $\mathcal{C}^M_q(\Omega)$
which minimize $D_\Omega$,
where $\Omega$ is the unit disk of the plane. }
Following \cite{LO}, we know that the main difficulty of this constrained shape
optimization problem comes from its great number of local minima. In order to tackle this
difficulty, we introduce a discretization of the problem with few parameters which
makes it possible to perform a stochastic optimization.

As in \cite{LO}, we parametrize optimal graphs as the convex hull of
a set of points. Consider a sampling $C_1,\ldots, C_n$ of  the unit circle $\partial \Omega$  made of $n$ points and let
 $\Omega_n\subset\Omega$ be the convex hull of this sampling. We introduce the cylindrical parametrization $\Phi_{M,q}$, defined for {
$(r,\theta,z) \in [0,1]\times[ 0,2 \pi ] \times  [0,1]$,}  by
$$\Phi_{M,q}(r,\theta,z) := (r \cos (\theta),r \sin (\theta),zM-q(r^2-1)/2).$$
If $\{P_1,\dots , P_m \}$ are $m$ points of {
$[0,1]\times [ 0,2 \pi ] \times  [0,1]$, we consider}
$$ \mathcal{G}_{P_1,\dots , P_m} := \text{Co}(\Omega_n  , \Phi_{M,q}(P_1),\dots,\Phi_{M,q}(P_m)) \setminus \Omega_n  ,$$
which is the convex-hull of the union of the points $\Phi_{M,q}(P_1),\dots,\Phi_{M,q}(P_m), C_1,\ldots, C_n$, minus
 $\Omega_n$.   {   $ \mathcal{G}_{P_1,\dots , P_m}$ is the polygonal graph} of a concave
function on $\Omega_n $. Moreover, if we denote by $v_{P_1,\dots , P_m}$
this  associated function, we have that
$$ u_{P_1,\dots , P_m}(x) := v_{P_1,\dots , P_m}(x) + q(|x|^2-1)/2,\quad x\in\Omega_n$$
is $q$-concave and has values in $[0,M]$. Conversely, every $q$-concave
function on $\Omega$ with values in  $[0,M]$ can be approximated by this procedure.

Let us focus now on the cost function evaluation, that is, on  the approximation of
\begin{equation*}
    D_{\Omega_n}(u_{P_1,\dots , P_m})=\int_{\Omega_n}\frac{dx}{1+|\nabla u_{P_1,\dots , P_m}(x)|^2}.
\end{equation*}
First, we observe that the situation is more complicated than the classical case
$q=0$ studied in \cite{LO}. As a matter of fact, the computation of
$D_{\Omega_n}(u_{P_1,\dots , P_m})$
does not reduce to a purely geometrical integral since $u_{P_1,\dots , P_m}$ is
not piecewise linear anymore. To provide a precise estimate of the previous integral,
we notice that $u_{P_1,\dots , P_m}$ is quadratic on every triangle $\tau$ obtained as
the projection on $\Omega$ of one triangular face of $\mathcal{G}_{P_1,\dots , P_m}$.
Moreover the integral
$$\int_{\tau}\frac{dx}{1+|\nabla u_{P_1,\dots , P_m}(x)|^2}$$
can be approximated by a Gauss quadrature formula of order $d $ if we provide the evaluation
of $u_{P_1,\dots , P_m}$ at every control points of the quadrature. We summarize
the different steps required for one cost function evaluation in Algorithm \ref{algo:cost},
choosing a Gauss quadrature { with $n_c$ } control points.
\begin{algorithm}[t]
    \caption{Cost evaluation.}
    \label{algo:cost}
    \begin{description}
        \item[Input] $M>0$, $q >0$,  a sampling of  $\partial \Omega$ { with points
        $\{ C_1, \dots C_n \}$, } and
        parameters $(r_1,\theta_1,z_1), \dots, (r_m,\theta_m,z_m)$
        \item[Convex Hull] Compute the convex hull of $\{ C_1, \dots C_n \} \cup
        \{ \Phi_{M,q}(P_1),\dots,\Phi_{M,q}(P_m) \}$ (complexity of order $(m+n) \log (m+n)$)
        \item[Triangulation] Project every triangular face { on $ \Omega$} to obtain a triangulation $\mathcal{T}$ of the convex hull of $\{ C_1, \dots C_n \}$.
        \item[Gauss control points]  For every $\tau \in  \mathcal{T}$, compute the
        associated $n_c$ control points $\{ Q_1^\tau, \dots Q_{n_c}^\tau \}.$
        \item[Evaluation]  For every $\tau \in \mathcal{T}$, for every control point $Q^\tau$,
        compute $\nabla u_{P_1,\dots , P_m}(Q^\tau)$. This step is reduced to a linear interpolation
        and a quadratic evaluation.

        \item[Output] return the Gauss quadrature approximation based on the control
        points $(Q^\tau_l)_{1 \leq l \leq {n_c},\, \tau \in \mathcal{T}}$.
    \end{description}
\end{algorithm}

\begin{figure}[ht]
    \centering
    \begin{tabular}{r r}
        \includegraphics[width=\widthh cm]{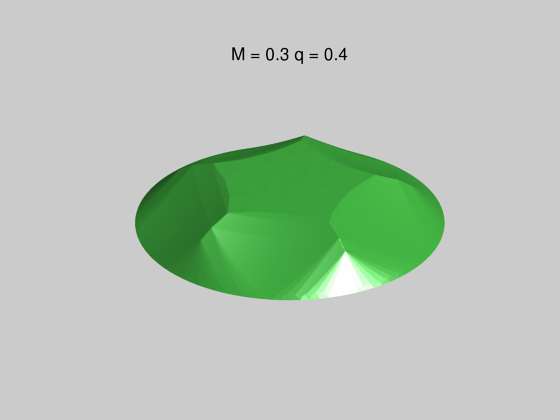}
        &\includegraphics[width=\widthh cm]{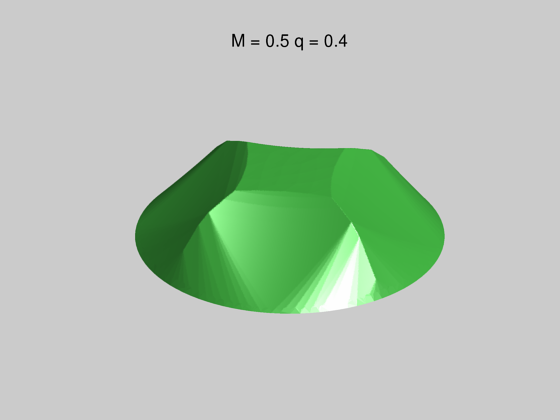}\\
        \includegraphics[width=\widthh cm]{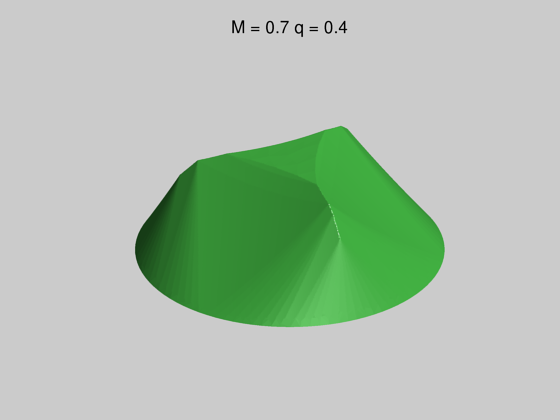}
        &\includegraphics[width=\widthh cm]{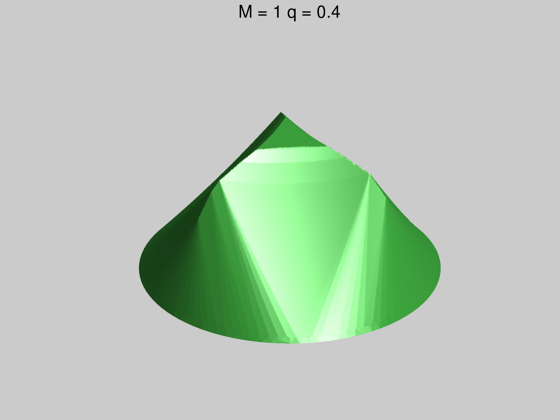}
    \end{tabular}
    \caption{Optimal computed profiles for $q=0.4$ and $M=0.3,\, 0.5, 0.7,\, 1$}
    \label{fig:f0}
\end{figure}

Based on this discretization involving only a few parameters $m=50$ (that is $150$
parameters), { $d=10$, $n_c=100$ and  $n=100$} it has been possible to perform in five hours $10^7$  evaluations of the discretized cost
function on a standard recent laptop. We used the algorithm
\texttt{adaptive\_de\_rand\_1\_bin\_radiuslimited} provided by the \texttt{BlackBoxOptim}
library (see \cite{BlackBoxOptim}). We represent in  Figure \ref{fig:f0}, several
$q$-concave optimal profiles for the same value $q=0.4$. The observed  qualitative
behavior is analogous to the one of the solutions computed in \cite{LO} in the case $q=0$:

\begin{itemize}
    \item Optimal graphs touch the constrained height hyperplane on a curvilinear
    polygon which seems  to be regular. By the way, notice that for $q>0$, there
    is no flat upper contact anymore. This flat part is replaced by a parabola
    when $q>0$,
    \item singular arcs, raising from the vertices of the upper polygon, can be
    observed in the graph,
    \item non strictly { concave } parts of the graph for $q=0$ are substituted by
    parabolic  patches.
\end{itemize}


 \section{Appendix: single shock and $q$-concave profiles}

The single shock condition reflects  the physical fact that every fluid particle hits the body at most once.
We shall deduce a corresponding geometric constraint on the body profile. See also \cite{BFK1, clr2, P1}.

Let $\Omega\subset \mathbb{R}^{n}$ a bounded convex open set and let $u:\Omega\to \mathbb{R}$ an a.e. differentiable function. We consider  a single point particle, moving in $\text{epi } u$ and approaching the graph of $u$ vertically downwards (i.e., along the direction of the coordinate vector $\mathbf{e}_{n+1}$) with constant nonnull velocity $\vv=-v\mathbf{e}_{n+1}$, $v>0$. We suppose that the particle hits the graph of $u$ elastically     at the point $(x_{0},u(x_{0}))\in\mathbb{R}^{n+1}$, such that $\nabla u(x_{0})$ exists. Furthermore we assume that the particle is reflected according to the usual laws of reflection.
Denoting by $\nu_0$ the outward normal unit vector at $(x_{0},u(x_{0}))$, i.e.,
\[{\nu}_{{0}}:=\left(\frac{-\nabla u(x_{0})}{\sqrt{1+|\nabla u (x_{0})|^{2}}},\frac{1}{\sqrt{1+|\nabla u (x_{0})|^{2}}}\right),\]
we let  ${\tau}_{{0}}$ be a vector lying in the subspace of $\mathbb{R}^{n+1}$ generated by ${\vv}$ and ${\nu}_{{0}}$, such that ${\nu}_{{0}}\cdot {\tau}_{{0}}=0.$
We denote by $z(t)=({x}(t),y(t))\in\mathbb{R}^{n+1}$, $t>0$, the position of the particle after the shock, occurring at $t=0$. If we consider  the components of the velocity vector $z'(t)$ along $\nu_0$ and $\tau_0$,
 according to the laws of reflection we have to impose
\[
\left\{
\begin{array}{rl}
\left[z'(t)\cdot{\nu}_{{0}}\right]{\nu}_{{0}}&=-\left({\vv}\cdot {\nu}_{{0}}\right){\nu}_{{0}}\\
\left[z'(t)\cdot{\tau}_{{0}}\right]{\tau}_{{0}}&=\left({\vv}\cdot {\tau}_{{0}}\right){\tau}_{{0}},
\end{array}
\right.
\]
that is,
\[
\left\{
\begin{array}{rl}
\left[z'(t)\cdot{\nu}_{{0}}\right]{\nu}_{{0}}&=-\left({\vv}\cdot {\nu}_{{0}}\right){\nu}_{{0}}\\
 z'(t)-\left[z'(t)\cdot {\nu}_{{0}}\right]{\nu}_{{0}}&={\vv}-\left({\vv}\cdot {\nu}_{{0}}\right){\nu}_{{0}}.
\end{array}
\right.
\]
So we obtain that
\[\begin{aligned}
z'(t)&={\vv}-2\left({\vv}\cdot {\nu}_{{0}}\right){\nu}_{{0}}
={\vv}+\frac{2v}{\sqrt{1+|\nabla u (x_{0})|^{2}}}{\nu}_{{0}}=	\vv+\dfrac{2v}{1+\left|\nabla u (x_{0})\right|^{2}}\left(-\nabla u (x_{0}),1\right)\\&=\left(-2\,\frac{\nabla u (x_{0})}{1+|\nabla u (x_{0})|^{2}}\,v\,,\,\frac{1-|\nabla u (x_{0})|^{2}}{1+|\nabla u (x_{0})|^{2}}\,v\right).
\end{aligned}\]
The trajectory of the particle after the collision is therefore described for $t>0$ by
\[
\left\{
\begin{array}{rl}
{x}(t)&=x_{0}-2\,\dfrac{\nabla u (x_{0})}{1+\left|\nabla u (x_{0})\right|^{2}}\,vt\\
{y}(t)&=u(x_{0})+\dfrac{1-\left|\nabla u (x_{0})\right|^{2}}{1+\left|\nabla u (x_{0})\right|^{2}}\,vt.
\end{array}
\right.
\]
The single shock condition at $(x_{0},u(x_{0}))$, which is
$u({x}(t))\le y(t)$ for any $t>0$, is then given by
\[u\left(x_{0}-2\,\dfrac{\nabla u (x_{0})}{1+\left|\nabla u (x_{0})\right|^{2}}\,vt\right)\le u(x_{0})+\dfrac{1-\left|\nabla u (x_{0})\right|^{2}}{1+\left|\nabla u (x_{0})\right|^{2}}\,vt.\]
If we rescale the time by  letting
$\tilde{t}_{x_{0}}:=\tfrac{2vt}{1+|\nabla u(x_{0})|^{2}}$,
the above inequality rewrites as follows
\[u\left(x_{0}-\tilde{t}_{x_{0}}\nabla u(x_{0})\right)\le u(x_{0})+\dfrac{\tilde{t}_{x_{0}}}{2}\left(1-|\nabla u(x_{0})|^{2}\right).\]

The above discussion motivates the following
\begin{definition} Let $\Omega$ be an open bounded convex subset of $\mathbb{R}^{n}$. We say that $u\colon\Omega\to\mathbb{R}$ is a \textit{single shock function on $\Omega$} if $u$ is  a.e. differentiable in $\Omega$ and
\begin{equation*}\label{single}u\left(x-\tau\nabla u(x)\right)\le u(x)+\dfrac{\tau}{2}\left(1-|\nabla u(x)|^{2}\right)\end{equation*}
for a.e. $x\in\Omega$ and for every $\tau>0$ such that  $x-\tau\nabla u(x)\in\Omega$.
\end{definition}

Next we discuss the relation between single shock and $q$-concave profiles. We start by recalling the definition of $q$-concavity.

\begin{definition}($q$-concave function) Let $\Omega$ be 
a convex subset of $\mathbb{R}^{n}$ and $q\ge 0$.  A function $u\colon \Omega\to \mathbb{R}$ is said to be  \textit{$q$-concave on $\Omega$} if  the map
$x\mapsto u(x)-\tfrac q2 \left|x\right|^{2}$  is  concave  on $\Omega$.  
Equivalently, $u$ is $q$-concave on $\Omega$ if and only if
\[u\left(\lambda x+\left(1-\lambda\right)y\right)\ge \lambda u(x)+(1-\lambda)u(y)-\dfrac{q}{2}\lambda(1-\lambda)\left|x-y\right|^{2}\] for every $x,y\in \Omega$ and for every $\lambda\in [0,1]$.  \end{definition}

\begin{lemma}\label{qshock}
Let  $q\ge 0$ and $\Omega\subset\mathbb{R}^{n}$ be a bounded convex open set. If $u\colon \Omega\to\mathbb{R}$ is a $q$-concave function on $\Omega$, and $q\,\text{diam}(\Omega)\le 2$, than $u$ has the single shock property on $\Omega$. In particular, if $u$ is concave then it is single shock in $\Omega$.
\end{lemma}
\begin{proof}[Proof.] Let $x\in\Omega$ be such that $\nabla u(x)$ exists, and let $\tau>0$ be such that $x-\tau\nabla u(x)\in\Omega$. Using the $q$-concavity of $u$, the fact that if $x-\tau\nabla u(x)\in\Omega$ then $\tau|\nabla u(x)|\le \text{diam}(\Omega)$, we have \begin{align*}
u(x-\tau\nabla u(x))&\le u(x)+\tau\left(-|\nabla u(x)|^{2}+\dfrac{q\tau}{2}|\nabla u(x)|^{2}\right)\\
& \le u(x)+\tau\left(-|\nabla u(x)|^{2}+\dfrac{q}{2}|\nabla u(x)|\,\text{diam}(\Omega)\right)\\
&\le u(x)+\tau\left(-|\nabla u(x)|^{2}+|\nabla u(x)|\right)\\
&= u(x)+\dfrac{\tau}{2}\left(1-|\nabla u(x)|^{2}\right)-\dfrac{\tau}{2}\left(|\nabla u(x)|-1\right)^{2}\\
&\le u(x)+\dfrac{\tau}{2}\left(1-|\nabla u(x)|^{2}\right),
\end{align*}
where we made use of the assumption $q\,\text{diam}(\Omega)\le 2$.
\end{proof}
\begin{remark} \rm The inequality $q\,\text{diam}(\Omega)\le 2$ is sharp. Indeed, if $\Omega$ is a ball, centered at the origin,  and $q\,\text{diam}(\Omega)>2$, then the function $\wp_{q}\colon \Omega\to \mathbb{R}$ defined by $\wp_{q}(x):=\tfrac q2|x|^{2}$  is \text{not} a single-shock function on $\Omega.$
\end{remark}

Existence of minimizers of the resistance functional on  $\mathcal{C}_q^M(\Omega)$ follows the standard arguments.

\begin{lemma}\label{compactness} Let $\Omega$ be an open bounded convex subset of $\mathbb{R}^{n}$. Let $M>0$ and $q\ge 0$. Then for every $p\in[1,\infty)$ the class $\mathcal{C}_q^M(\Omega)$ is compact with respect to the strong topology of $W^{1,p}_{\emph{loc}}(\Omega)$.
\end{lemma}
\begin{proof}
First of all, a concave function $v$ on $\Omega$ taking values in $[0,M]$ satisfies, for every $K\subset\subset\Omega$,
\[\left|v\left(z_{1}\right)-v\left(z_{2}\right)\right|\le\frac{2M}{\emph{dist}\,\left(K,\partial\Omega\right)}\left|z_{1}-z_{2}\right|\quad\text{ for every }z_{1},z_{2}\in K.\]
Then, if $R>0$ is such that  $\Omega\subset B_0(R)$,     a $q$-concave function is Lipschitz continuous on any open subset $K$, compactly contained in $\Omega$, with Lipschitz constant not exceeding
$\tfrac{2M}{\emph{dist}(K,\partial\Omega)}+qR$.

Let $\left(u_{n}\right)_{n\in\mathbb{N}}$ be a sequence of elements of $\mathcal{C}_q^M(\Omega)$. We shall prove that there exists a strictly increasing sequence of natural numbers $\left(n_{k}\right)_{k\in\mathbb{N}}$ and $u\in \mathcal{S}_{M,q}(\Omega)$ such that
\[u_{n_k}\rightarrow u \text{ in }L^p\left(\Omega\right)\text{ e }  \nabla u_{n_k}\rightarrow  \nabla u\text{ in }L^p(K) \text{ for every }K\subset\subset \Omega.\]

\noindent The sequence $\left(u_{n}\right)_{n\in\mathbb{N}}$ is  equi-bounded and  equi-Lipschitz  on every $K\subset\subset \Omega$.  By Ascoli-Arzel\`a theorem, $\left(u_{n}\right)_{n\in\mathbb{N}}$ admits  a convergent subsequence in $C(K)$, for every $K\subset\subset \Omega.$  By a diagonal argument we may obtain the existence of a strictly increasing sequence of natural numbers $\left(n_{k}\right)_{k\in\mathbb{N}}$ and of a function $u\in C(\Omega)$, such that $u_{n_{k}}\to u$  uniformly on each $K\subset\subset \Omega$. \EEE Since
\begin{align*}
u\left(\lambda x+\left(1-\lambda\right)y\right)&=\lim\limits_{k\to+\infty}u_{n_{k}}\left(\lambda x+\left(1-\lambda\right)y\right)\\
&\ge \lim\limits_{k\to+\infty}\left[\lambda u_{n_{k}}(x)+\left(1-\lambda\right)u_{n_{k}}(y)-\dfrac{q}{2}\lambda\left(1-\lambda\right)\left| x-y\right|^{2}\right]\\&=\lambda u(x)+\left(1-\lambda\right)u(y)-\dfrac{q}{2}\lambda\left(1-\lambda\right)\left| x-y\right|^{2}
\end{align*}
for every $x,y\in\Omega$ and for every $\lambda\in[0,1]$, $u$ is $q$-concave on $\Omega$. Moreover, since $u_{n_{k}}(x)\in [0,M]$ for every $x\in \Omega$ and for every $k\in\mathbb{N}$, we have $u(x)\in[0,M]$ for every $x\in \Omega$. Thus $u\in\mathcal{C}_q^M(\Omega)$.
Now, since $\Omega$ is bounded and $\left(u_{n_{k}}\right)_{k\in\mathbb{N}}$ is an equi-bounded subsequence, by dominated convergence we infer that $u_{n_{k}}\to u$ in $L^{p}(\Omega).$
In order to conclude we have to show that $\nabla u_{n_{k}}\to \nabla u$ in $L^{p}(K)$ for every $K\subset\subset \Omega$. Since $\left(u_{n_{k}}\right)_{k\in\mathbb{N}}$ is equi-Lipschitz continuous on each $K\subset\subset \Omega$, we have that $\left( \nabla u_{n_{k}}\right)_{k\in\mathbb{N}}$ is equi-bounded on each $K\subset\subset \Omega$. So,  it suffices to prove that
\[
\nabla u_{n_{k}}(x)\to\nabla u(x)\text{ for a.e. }x\in \Omega.
\]
 Let $i\in\left\{1,\ldots,n\right\}$ and let $x\in \Omega$ be a fixed point where all $u_{n_{k}}$ $(k\in\mathbb{N})$ and $u$ are differentiable (almost every point of $\Omega$ meets this requirement). Denoting by ${e}_{i}$ the  $i$-th vector of the standard basis in $\mathbb{R}^{n}$ and letting $\varphi_{n_k}(x):=u_{n_k}(x)-\tfrac q2x^2$, since the functions $t\mapsto \varphi_{{n_{k}}}(x+te_{i})$ are concave, there exists $\varepsilon_{0}=\varepsilon_{0}(i,x)>0$ such that, for every $\varepsilon \in(0,\varepsilon_{0})$
\[\dfrac{\varphi_{{n_{k}}}(x+\varepsilon e_{i})-\varphi_{{n_{k}}}(x)}{\varepsilon}\le \partial_{i}\varphi_{{n_{k}}}(x)\le\dfrac{\varphi_{{n_{k}}}(x-\varepsilon e_{i})-\varphi_{{n_{k}}}(x)}{-\varepsilon}\]
from which, adding $qx_{i}$  and taking into account that  $\partial_{i}\varphi_{u_{n_{k}}}(x)=\partial_{i}u_{n_{k}}(x)-qx_{i}$, we have
\[\dfrac{u_{n_{k}}(x+\varepsilon e_{i})-u_{n_{k}}(x)}{\varepsilon}-\dfrac{q\varepsilon}{2}\le \partial_{i}u_{n_{k}}(x)\le\dfrac{u_{n_{k}}(x-\varepsilon e_{i})-u_{n_{k}}(x)}{-\varepsilon}+\dfrac{q\varepsilon}{2}.\] Passing to the limit as $k\to +\infty$, for every $\varepsilon\in(0,\varepsilon_{0})$ we obtain
\[\dfrac{u(x+\varepsilon e_{i})-u(x)}{\varepsilon}-\dfrac{q\varepsilon}{2}\le \liminf\limits_{k}\partial_{i} u_{n_{k}}(x)\le\limsup\limits_{k}\partial_{i}u_{n_{k}}(x)\le \dfrac{u(x-\varepsilon e_{i})-u(x)}{-\varepsilon}+\dfrac{q\varepsilon}{2}.\] Passing now to the limit as $\varepsilon\to 0$ we have
\[\partial_{i}u(x)\le \liminf\limits_{k}\partial_{i} u_{n_{k}}(x)\le\limsup\limits_{k}\partial_{i}u_{n_{k}}(x)\le \partial_{i}u(x),\] that is,  $\lim\limits_{k\to+\infty}\partial _{i}u_{n_{k}}(x)=\partial_{i}u(x).$
\end{proof}

\begin{corollary}
 Let $\Omega$ be an open bounded convex subset of $\mathbb{R}^{n}$. Let $M>0$ and $q\ge 0$. The resistance functional $D_\Omega$ admits a minimizer on $\mathcal{C}_q^M(\Omega)$.
\end{corollary}
\begin{proof}
Notice that, by dominated convergence, functional $D_\Omega$ is continuous with respect to the a.e. convergence of gradients.
\end{proof}

\subsection*{Acknowledgements}
E.M.  is member of the
GNAMPA group of the Istituto Nazionale di Alta Matematica (INdAM).


\begin{thebibliography}{99}

\bibitem[BBO]{BlackBoxOptim} \texttt{BlackBoxOptim.jl}, a global optimization
framework for Julia, \url{https://github.com/robertfeldt/BlackBoxOptim.jl}.

\bibitem[BFK1]{BFK1} F. Brock, V. Ferone, B. Kawohl, {\it  A Symmetry Problem in the Calculus of Variations}, Calc. Var. Partial Differential Equations {\bf 4} (1996), 593--599.

\bibitem[B]{B} G. Buttazzo, {\it A survey on the Newton problem of optimal profiles}. In `Variational Analysis and Aerospace Engineering',
Volume 33 of the series Springer Optimization and Its Applications (2009), 33-48.



\bibitem[BFK2]{BFK2} G. Buttazzo, V. Ferone, B. Kawohl, {\it Minimum problems over sets of concave functions and related questions}, Math. Nachr. {\bf 173} (1995), 71--89.

\bibitem[BG]{BG} {G. Buttazzo, P. Guasoni},  {\it Shape optimization problems over classes of
convex domains}, J. Convex Anal. {\bf  4} (1997), 343--351.


\bibitem[BK]{BK} G. Buttazzo, B. Kawohl, {\it On Newton's problem of minimal resistance}, Math. Intelligencer {\bf 15} (4) (1993),  7--12.



\bibitem[CL1]{clr3}
M. Comte, T. Lachand-Robert,
\newblock{\em Newton's problem of the body of minimal resistance under a single-impact assumption},
\newblock Calc. Var. Partial Differential Equations {\bf 12} (2001), 173--211.

\bibitem[CL2]{clr2}
M. Comte, T. Lachand-Robert,
\newblock{\em Existence of minimizers for the Newton's problem of the body of minimal resistance under a single-impact assumption},
\newblock J. Anal. Math. {\bf 83} (2001),  313--335.




\bibitem[G]{G} H. H. Goldstine,  A history of the calculus of variations from
the 17th through the 19th Century. Heidelberg: Springer-
Verlag (1980).


%




\bibitem[LO]{LO}
T. Lachand-Robert, \'E. Oudet,
\newblock{\em Minimizing within convex bodies using a convex hull method},
\newblock SIAM J. Optim. {\bf 16} (2005), pp. 368--379.

%

\bibitem[LP]{LP} T. Lachand-Robert, M. A. Peletier, {\it Newton's Problem of the Body of Minimal Resistance in the Class of Convex Developable Functions}
, Math. Nachr. {\bf 226} (2001), 153--176.


\bibitem[M]{M} P. Marcellini, {\it Nonconvex integrals of the calculus of variations}. In `Methods of Nonconvex Analysis' (Varenna, 1989), Lecture Notes in Math. 1446, Springer-Verlag, Berlin (1990), 16--57.


\bibitem[$\text{MMOP}$]{MMOP} E. Mainini, M. Monteverde, E. Oudet, D. Percivale, {\it  Newton's aerodynamic for non convex bodies}, to appear on Rend. Lincei Mat. Appl.




\bibitem[P1]{P1} A. Plakhov, {\it The  problem of minimal resistance for functions and domains.}   SIAM J. Math. Anal. {\bf 46}  (2014),    2730-2742.

\bibitem[P2]{P2} A. Plakhov, {\it Newton's problem of minimal resistance under the single impact assumption.}  Nonlinearity {\bf 29} (2016), 465-488.


\end{thebibliography}
\end{document}